\documentclass[a4paper, english, 12pt]{amsart}
\usepackage{geometry}
\usepackage[T1]{fontenc}
\usepackage[utf8]{inputenc}	
\usepackage{babel}
\usepackage{amsmath}
\usepackage{amssymb}
\usepackage{amsthm}
\usepackage{textcomp}
\usepackage{enumitem}
\usepackage{dsfont}
\usepackage{mathrsfs}
\usepackage{latexsym}
\usepackage{xcolor}
\usepackage{graphicx}
\usepackage[]{mdframed}
\usepackage{cite}
\usepackage{color}
\usepackage[colorlinks=true, linkcolor=blue]{hyperref}

\numberwithin{equation}{section}
\numberwithin{figure}{section}
\def\R{{\mathbb R}}
\def\C{{\mathbb C}}

\def\T{{\mathbb T}}

\def\N{{\mathbb N}}

\newcommand{\ii}{\mathrm{i}}

\DeclareMathOperator{\Id}{Id}

\def\build#1_#2^#3{\mathrel{
\mathop{\kern 0pt#1}\limits_{#2}^{#3}}}
\def\td_#1,#2{\mathrel{\mathop{\build\longrightarrow_{#1\rightarrow #2}^{}}}}
\DeclareFontFamily{U}{MnSymbolC}{}
\DeclareSymbolFont{MnSyC}{U}{MnSymbolC}{m}{n}
\DeclareFontShape{U}{MnSymbolC}{m}{n}{
    <-6>  MnSymbolC5
   <6-7>  MnSymbolC6
   <7-8>  MnSymbolC7
   <8-9>  MnSymbolC8
   <9-10> MnSymbolC9
  <10-12> MnSymbolC10
  <12->   MnSymbolC12}{}
\DeclareMathSymbol{\intprod}{\mathbin}{MnSyC}{'270}

\newtheorem{theorem}{Theorem}

\newtheorem{proposition}{Proposition}
\newtheorem{lemma}{Lemma}
\newtheorem{remark}{Remark}
\newtheorem{definition}{Definition}

\title[An Explicit Formula for the Benjamin--Ono Hierarchy with Applications]{An Explicit Formula for the Benjamin--Ono Hierarchy with Applications to Traveling Waves and Zero-Dispersion Limits} 

\date{\today}

\author[P.~G\'erard]{Patrick G\'erard} \address{Laboratoire de Mathématiques d'Orsay,  Université Paris-Saclay, Orsay, France} \email{patrick.gerard@universite-paris-saclay.fr}

\author[J.~He]{Jiao He} \address{Laboratoire de Mathématiques d'Orsay,  Université Paris-Saclay, Orsay, France} \email{jiao.he@universite-paris-saclay.fr}



\begin{document}

\maketitle

\begin{abstract}
In this paper, we first extend the explicit formula \cite{gerard2023explicit} for the classical Benjamin--Ono equation to each flow of the Benjamin--Ono hierarchy on line. We then use this representation to derive two main applications. First, we obtain a complete classification of traveling wave solutions for all higher-order flows in the hierarchy.
Second, we analyze the zero-dispersion limit for the corresponding small-dispersion flows. For every fixed time $t\in\mathbb R$, we prove that, at any time, the solution converges weakly in $L^2(\mathbb R)$ as the dispersion parameter tends to $0$, and we  provide a geometric characterization of the limit in terms of an alternating sum, which yields the higher-order analogue of the  formula obtained in \cite{miller2011zero}, \cite{Gerard2025small}  for the Benjamin--Ono equation.
\end{abstract}

\tableofcontents

\section{Introduction}
The Benjamin–Ono (BO) equation was introduced independently by Benjamin \cite{benjamin1967internal} and
Davis– Acrivos \cite{davis1967solitary} in 1967, and was later further developed by Ono \cite{ono1975algebraic} in 1975. 
It serves as a model for long, unidirectional internal gravity waves in a two-layer fluid with infinite depth, see e.g. the book \cite{klein2021} and \cite{paulsen2024justification} for a recent derivation. The equation is given by
\begin{equation}\label{BO}
	\partial_tu- \partial_x |D| u + \partial_x (u^2)= 0, \quad  u(0, x) = u_0(x),	
	\tag{BO}
\end{equation}
where $u=u(t,x)$ is a real--valued function and $|D|$ denotes the Fourier multiplier $\mathcal{F}(|D|u)(t,\xi)=|\xi|\mathcal{F}(u)(t,\xi)$. 
The well-posedness theory for the \eqref{BO} equation in Sobolev spaces was pioneered by Saut \cite{saut1979}, while the sharp result was recently established by Killip and Vişan \cite{killip2024sharp}, who proved that the equation is well-posed in $H^s(\mathbb{R})$ for all $s > -1/2$. For a comprehensive overview of related results, we refer the reader to \cite{killip2024sharp} and the book by Klein and Saut \cite{klein2021}. 
Beyond well-posedness, significant progress has been made in characterizing the long-term asymptotic behavior of solutions. On the line, the soliton resolution conjecture has been proved \cite{blackstone2025benjamin, gassot2026proof} for all $H^1(\mathbb{R})$ initial data satisfying sufficient decay conditions at infinity, by using an \textit{explicit formula} discovered by the first author \cite{gerard2023explicit}. This formula is indeed an expression of the complete integrability, which is a fundamental property of the \eqref{BO} equation. Another consequence of the complete integrability is that, it admits an infinite number of conserved quantities $\{E_n\}_{n \geq 0}$; see Nakamura \cite{nakamura1979backlund} and Fokas-Fuchssteiner \cite{ fokas1981hierarchy} for a proof on the real line. The BO hierarchy consists of the evolution equations associated to the conservation laws \cite{matsuno1984bilinear},
\begin{equation}\label{BOhier}
	\partial_t u = \partial_x (\nabla E_n(u)).
\end{equation}
In particular, the \eqref{BO} equation corresponds to the second non-trivial flow, and can be written in Hamiltonian form \cite{gerard2021integrability, sun2021complete}
\begin{equation}\label{eq:BOEner}
	\partial_t u = \partial_x (\nabla E_1(u)), \qquad E_1 (u):= \frac{1}{2} \langle |D|u, u \rangle_{H^{-1/2}, H^{1/2}} - \int_\R \frac{1}{3} u^3 dx
\end{equation}
where $\langle \cdot, \cdot \rangle_{H^{-1/2}, H^{1/2}}$ denotes the duality bracket.

Let us also mention some earlier results on the Cauchy problem for higher-order \eqref{BO} equations and related models. For the third equation in the BO hierarchy on the real line, Feng and Han \cite{fenghan} proved local well-posedness in $H^s(\R)$ for $s\ge 4$. For a more general third-order \eqref{BO} equation, Molinet and Pilod \cite{molinet2012global} established global well-posedness in $H^s(\R)$ for $s\ge 1$. For fourth-order BO-type equations, Tanaka \cite{tanaka2021local} proved local well-posedness in $H^s(\R)$ for $s>\frac{7}{2}$, and in the integrable case deduced global well-posedness in $H^s(\R)$ for $s\ge 4$.

Rather than developing a general Cauchy theory for the whole hierarchy, the emphasis of the present paper is different and relies on the integrable structure. Our purpose is threefold. First, we derive a general \emph{explicit formula} for the solutions of each flow of the BO hierarchy on the real line, extending the previously well-known explicit formula for the classical \eqref{BO} equation \cite{gerard2023explicit}. Second, as a first application of this formula, we provide a complete classification of traveling waves for every flow in the hierarchy. Third, we apply this formula to study the zero-dispersion limit for the BO hierarchy and obtain a precise description of the weak limit.

Before presenting our main results, we review the Lax pair structure for \eqref{BO} equation and the BO hierarchy. 

\subsection{Notation}\label{subsection:notation}
Throughout the paper, we will use the following closed subspace of $L^2 (\R)$,
$$
L^2_+(\R):=\left\{ f\in L^2(\R): \forall \xi <0, \hat f(\xi)=0 \right\} \ ,
$$
where $\hat f$ denotes the Fourier transform. 
We recall that $L^2_+(\R)$ identifies to holomorphic fonctions $f$ on the complex upper half plane $\C_+ := \{z \in \C : \mathrm{Im}(z) >0  \}$ such that
$$
\sup_{y\,>\,0}\int_\R \left|f\left(x+ \ii y\right)\right|^2\, dx\ <+\infty \ .
$$ 
The orthogonal projector from $L^2(\R)$ onto $L^2_+(\R)$ is given by 
$$
\widehat{\Pi f}(\xi) : = \textrm{1}_{\xi \geq 0} \hat f (\xi). 
$$
On $L^2_+(\mathbb{R})$, we introduce the infinitesimal generator of the adjoint Lax–Beurling semi–group, namely the operator $X^*$ such that
$$
\forall \eta \geq 0, \quad S(\eta)^* = e^{- \ii \eta X^*}.$$
Its operator domain consists of those functions $f \in L^2_+(\mathbb{R})$ such that the restriction of $ \widehat{f}$ to the half-line $ (0, +\infty) $ belongs to the Sobolev space $H^1(0, +\infty)$, i.e.
$$
\mathrm{Dom} (X^*) = \{f \in L_+^2 (\R) : \hat{f}_{|]0, +\infty [ } \in H^1 (0, +\infty)  \}.
$$
Moreover, 
$$
\widehat{X^*f}(\xi) = \ii \frac{d}{d\xi}[\widehat{f}(\xi)] \mathbf{1}_{\xi > 0}.$$
For every $f \in \mathrm{Dom}(X^*) $, we define
$$
I_+(f) := \widehat{f}(0^+) = \lim_{\varepsilon \to 0^+} \left\langle f \mid \chi_\varepsilon \right\rangle_{L^2}, \quad \text{with} \quad \chi_\varepsilon := \frac{1}{1- \ii \varepsilon x} \in L^2_+ (\R).$$
Given $b\in L^\infty (\R)$, we define the Toeplitz operator on $L_+^2(\R)$ with symbol $b$ by
$$T_b : L^2_+ \to L^2_+, \;\;\;f \mapsto T_b (f) := \Pi (bf).$$
This operator is bounded, with $\|T_b\| = \| b \|_{L^\infty}$.\\
If $B$ is a bounded operator on $L^2_+(\R)$ and $(B_n)$ denotes a sequence of bounded operators on $L^2_+(\R)$, we write
$ B_n\xrightarrow[s]{} B$
to denote strong convergence, namely, $$\forall f\in L^2_+(\R),\quad \|B_nf-Bf\|_{L^2(\R)}\to 0.$$

Furthermore, we denote by $L^2_r$ the $\R$-Hilbert space $L^2(\R, \R)$, consisting of elements $u \in L^2$ which are real valued. 
For every $u \in L^2_r$, the
semi-bounded self-adjoint operator $L_{u} $ acting on the Hardy space $L^2_+$ is defined as
$$
\forall f \in \mathrm{Dom} (L_u) = H^1_+ : = H^1 (\R) \cap L^2_+ (\R), \quad L_u f := D f- T_u f \; \text{with }\; D := \frac{1}{\ii} \frac{d}{dx}.
$$
Alternatively, $|D| = \operatorname{H} \partial_{x}$, where $\operatorname{H}$ denotes the Hilbert transform. 

\subsection{The explicit formulae}
In \cite{gerard2023explicit}, the first author introduced an explicit formula for the solution of the equation \eqref{BO} with initial data $u_0 \in L^2_r (\R) \cap L^\infty (\R)$. This  formula is based on identities between the operators from the Lax pair and the Lax–Beurling semi–group on the Hardy space. Our first main result generalizes this explicit formula to the entire BO hierarchy \eqref{BOhier}.
\begin{theorem} \label{mainBOhier}
Let $n \in \mathbb{N}$ and $u_0 \in H^s_r(\mathbb{R})$ with $s > n + \frac{1}{2}$. Let $ u \in C(\mathbb{R}, H^s_r(\mathbb{R}))$ be a solution to the $(n+1)$-the order flow of the BO hierarchy \eqref{BOhier} on the line $\mathbb{R} $ with $ u(0) = u_0 $.  
Then $u (t, x)= \Pi u(t, x) + \overline{\Pi u(t, x)}, \ x\in \R  $, with
\begin{equation}\label{BOhierexpfor}
\forall\; z \in \mathbb{C}_+, \quad \Pi u(t, z) = \frac{1}{2 \ii  \pi} \, I_+ \left[ \left( X^* -  (n+1) t L_{u_0}^n - z \mathrm{Id} \right)^{-1} \Pi u_0 \right].
\end{equation}
\end{theorem}

\begin{remark}\label{remark1}
For $n=1$, we recover exactly the explicit formula for the standard \eqref{BO} equation derived in \cite{gerard2023explicit}, 
\begin{equation*}
	\Pi u (t,z)=\frac{1}{2 \ii \pi} I_+\left[\left(X^* - 2 t L_{u_0} -z\mathrm{Id} \right)^{-1}\Pi u_0\right].
\end{equation*}
\end{remark}

\begin{remark}
Recently, Killip, Laurens and Visan \cite{killip2024sharp} provided an analogous explicit formula for higher flows in the BO hierarchy. While their formulation (see \cite[Theorem 6.1]{killip2024sharp}) is based on the generating function, our formula \eqref{BOhierexpfor} provides a direct representation for each individual flow in terms of powers of the Lax operator $L_{u_0}$.
\end{remark}

\begin{remark}\label{remarkchen}
Chen \cite{chen2025explicit} extended the first author's explicit formula of the solution to \eqref{BO} to the initial data $u_0 \in L^2_r (\R)$.  A similar extension for higher flows would require a further study of the operators $X^* - (n+1)tL_{u_0}^n$; we do not pursue this direction in the present paper.
\end{remark}

\subsection{Classification of traveling waves}\label{subsecclasstw}
The second purpose of this paper is to classify traveling-wave solutions of the $(n+1)$-th flow of the BO hierarchy on the line. A traveling wave is a solution of the form
$$
u(t,x)=u_0(x-ct),
$$
for some $u_0\in L^2 (\R)$ and some velocity $c\in\R$. 	For the Korteweg--de Vries equation, the classification of decaying traveling waves is classical. Indeed, inserting the ansatz $u(t,x)=Q(x-ct)$ reduces the equation to an ODE, whose solutions are exactly the well-known solitary waves
$ Q(x)=\frac{c}{2}\,\operatorname{sech}^2\! (\frac{\sqrt c}{2}(x-x_0)).$
For the \eqref{BO} equation, the situation is delicate because the dispersion is nonlocal.
The existence of explicit \eqref{BO} solitary waves goes back to Benjamin, Davis--Acrivos, and Ono \cite{benjamin1967internal,davis1967solitary,ono1975algebraic}, while the complete classification of \eqref{BO} traveling waves was proved by Amick and Toland \cite{AmickToland1991}; see also the first author's recent proof using Lax-pair in \cite[Theorem~5.4]{gerard2026lectures}.
More precisely, every nontrivial traveling wave of \eqref{BO} is a one-soliton of the form
\begin{equation}\label{eq:solitonBO}
	u(t,x)=R_p(x-c_pt),\quad R_p(y):=\frac{2\mathrm{Im}p}{|y+p|^2},\quad c_p:=\frac{1}{\mathrm{Im}p}, \quad p\in \C_+.	
\end{equation}

To the best of our knowledge, the classification of traveling waves for the BO hierarchy on the line has not been investigated before. We also refer to the work of Tzvetkov and Visciglia \cite{tzvetkov2013gaussian}, in which Gaussian measures associated with the higher-order conservation laws of the \eqref{BO} equation are constructed in the periodic setting. In contrast with their probabilistic perspective, our second theorem yields a deterministic counterpart on the line. Namely, we prove that the phenomenon exhibited in \eqref{eq:solitonBO} remains valid for the higher flows: every nontrivial traveling wave of the $(n+1)$-th flow is still a one-soliton with the same profile $R_p$, and only the velocity changes according to the flow. 

Our main theorem reads as follows.

\begin{theorem}[Classification of traveling waves for the BO hierarchy]\label{thm:classification}
Let $n \in \mathbb{N}$ and $u_0 \in H^s_r(\mathbb{R})$ with $s > n + \frac{1}{2}$. If $u$ is a traveling wave, namely if there exists $c\in\R$ such that
	$$
	u(t,x)=u_0(x-ct),
	\qquad (t,x)\in\R^2,
	$$
then either $u_0\equiv 0,$ or there exists a unique parameter $p\in\C_+$ such that
	$$
	u(t,x)=R_p(x-c_{n,p}t),
	\qquad
	R_p(y):=\frac{2\mathrm{Im} p}{|y+p|^2},\qquad
	c_{n,p}:=(-1)^{n+1}\frac{n+1}{(2\mathrm{Im} p)^n}.
	$$
\end{theorem}

\begin{remark}\label{rqBO3station}
In the case $n = 2$, let $u$ be a traveling wave of the form
$u(t,x)=Q(x- c_{2,p} t)$. Substituting $\partial_t u = -c_{2,p} \partial_x u$ with $c_{2,p} = -\frac{3}{4(\mathrm{Im}\, p)^2}$, into the third-order equation of the BO hierarchy \eqref{BO3}, we obtain the following stationary equation
\begin{equation}\label{BO3stat}
\partial_x^2 Q - c_{2,p} Q + \frac{3}{2} Q |D|Q + \frac{3}{2} \operatorname{H}(Q \partial_x Q) - Q^3 = 0.
\end{equation}
Theorem~\ref{thm:classification} then implies that any nontrivial solution of \eqref{BO3stat} in $H^s_r(\mathbb{R})$ must be of the form $R_p$.
\end{remark}

\begin{remark}\label{rqlouise}
The situation on the torus is different from that on the line. More precisely, in the periodic setting, Gassot \cite{gassot2021third} proved that the third-order \eqref{BO} equation admits more traveling waves in $L_{r,0}^2(\T)$ than the classical \eqref{BO} equation. By contrast, on the line, Theorem~\ref{thm:classification} shows that every nontrivial traveling wave remains a one-soliton.
\end{remark}

\begin{remark}\label{rqdiff}
Although the higher flows in the BO hierarchy are more complicated than the classical BO flow, the proof of Theorem~\ref{thm:classification} given in Section~\ref{sectiontw} is both different from and simpler than the previously known approaches to the standard BO equation, such as the analysis of Amick and Toland \cite{AmickToland1991} and the first author's Lax-pair proof in \cite[Theorem~5.4]{gerard2026lectures}. Instead, our argument is based on the explicit formula for the BO hierarchy and the soliton limit, which make the classification direct.
\end{remark}

\subsection{Zero dispersion limit}
Our third objective is to identify the zero dispersion limit of solutions of the BO hierarchy on the line. Before stating our main results, we first review the background of the zero dispersion problem for the standard \eqref{BO} equation. 

The \eqref{BO} equation with a small dispersion parameter $\varepsilon > 0$ is given by
\begin{equation}\label{BOeps}
	\partial_t u^\varepsilon- \varepsilon \partial_x |D| u^\varepsilon + \partial_x ((u^\varepsilon)^2)= 0, \quad  u^\varepsilon(0, x) = u_0(x).
	\tag{BO-$\varepsilon$}
\end{equation}
The study of the zero-dispersion limit explores the complex interplay between nonlinearity and dispersion as the dispersion parameter $\varepsilon$ tends to zero. 
The rigorous mathematical theory of zero-dispersion limits was pioneered by Lax and Levermore \cite{lax1983small} in their seminal work on the Korteweg--de Vries (KdV) equation. Using the scattering transform method, they demonstrated that, for special classes of
initial data under strong assumptions,  the zero-dispersion limit of the KdV equation exists in the weak $L^2$ sense and can be characterized as the solution of a variational problem. This framework was subsequently extended to higher-order by Venakides \cite{venakides1991korteweg}. Building upon the Lax-Levermore strategy, similar analyses were established for nonlinear Schrödinger equation (NLS) and modified Korteweg-de Vries (mKdV) equation \cite{jin1999semiclassical, ercolani2003zero, kamvissis2003semiclassical}.

For the \eqref{BOeps} equation, the first rigorous analysis was obtained by Miller and Xu \cite{miller2011zero}, who developed an appropriate analogue of the Lax--Levermore theory for bell-shaped initial data on the line. Unlike the KdV case, their formula for the BO weak limit is explicit. Specifically, by the method of characteristics, they solved the initial-value problem for the inviscid Burgers equation 
$u^\mathrm{B}_t+ \partial_{x} (u^\mathrm{B})^2 =0, u^\mathrm{B}(0,x)=u_0(x),$
where the solutions possesses an odd number of branches, which we denote by $u^B_0(t,x)<u^B_1(t,x)<\dots<u^B_{2J}(t,x)$.
Then, for initial data with bell-shaped
profiles, the solution of the \eqref{BOeps} equation converges weakly in $L^2(\mathbb{R})$ to a limit $ZD [u_0](t, x)$ defined explicitly by the following alternating sum,
\begin{equation}\label{BOlimit}
ZD[u_0](t, x):=\sum_{n=0}^{2J(t,x)}(-1)^n u_n^\mathrm{B}(t,x)
\end{equation}
This alternating sum formula was subsequently established for periodic bell-shaped data on the torus  by Gassot in \cite{gassot2023zero, gassot2023lax}.

Recently, unlike the Lax--Levermore theory, which requires strong assumptions on the initial data,
the discovery \cite{gerard2023explicit} of an \emph{explicit formula} for \eqref{BO} equation has led to more progress in the study of the zero dispersion limit. Using this tool, the first author gave a more efficient proof to generalize the results of Miller and Xu to all initial data in $L^2(\mathbb{R}) \cap L^\infty(\mathbb{R})$ on the line\cite{Gerard2025small}. This has been further extended in several directions. 
As mentioned in Remark \ref{remarkchen}, Chen generalized the explicit formula for \eqref{BO} with $u_0 \in L^2_r (\R)$, which allowed to obtain a weak convergence result at lower regularity \cite{chen2025explicit}. An analogous strategy was implemented by M{\ae}hlen \cite{maehlen2025zero} for the periodic case, generalizing the work of Gassot \cite{gassot2023zero} to all bounded initial data. Furthermore, Blackstone et al. \cite{blackstone2024benjamin} utilized the explicit formula to establish the leading-order asymptotics in $L^2$ strong convergence for certain rational initial data. Beyond the \eqref{BO} equation itself,  the Calogero–Moser derivative nonlinear Schrödinger equation
\cite{gerard2024calogero} admits an integrable structure and an explicit
formula close to that of the \eqref{BO}, Badreddine thus establishes a weak convergence result in the zero-dispersion limit for that equation using a similar approach \cite{badreddine2024zero}.

While the zero-dispersion limit of the standard \eqref{BO} equation has been extensively  studied, the corresponding problem for the higher-order flows in the BO hierarchy remains an open problem. It is natural to ask whether the alternating sum \eqref{BOlimit} extends to the whole hierarchy. To the best of our knowledge,	the only existing study is the work of Miller and Xu \cite{miller2012zeroBOhier}, which utilized inverse-scattering and a Lax--Levermore-type approach for admissible positive initial data. Our approach is of a different nature. Starting from the explicit formula established in Theorem~\ref{mainBOhier}, we treat each fixed 
$n+1$-st flow directly. In particular, we do not use scattering data and do not impose special shape assumptions on the initial profile, under the natural Sobolev regularity required to define the flow, we prove that 
$u^\varepsilon (t)$ converges weakly in $L^2 (\R)$ for every fixed $t$, and we identify the limit explicitly. 

For $n\in \mathbb{N}$, the corresponding small-dispersion $(n+1)$-st flow reads
\begin{align}\label{eq:BOneps}
	\partial_t u^\varepsilon = \partial_x (\nabla E_n^\varepsilon (u^\varepsilon)),  \qquad
	u^\varepsilon(0, x) = u_0(x),
\end{align}
where
$$
E_n^\varepsilon(u^\varepsilon) := \langle (\varepsilon D - T_{u^\varepsilon})^n \Pi u^\varepsilon, \Pi u^\varepsilon \rangle_{L^2}.$$
The conservation of the $L^2$ norm ensures that $\{u^\varepsilon\}_{\varepsilon > 0}$ is uniformly bounded in $L^2(\mathbb{R})$. By weak compactness, there exists a subsequence $\varepsilon_j \to 0$ such that $u^{\varepsilon_j}$ converges weakly to a limit. 
The following theorem provides 
a precise geometric characterization of this limit, which extends the alternating sum formula \eqref{BOlimit} to the entire hierarchy and a wide class of initial data.

\begin{theorem}\label{thm:zerolimit}
Let $n \in \N$ and $u_0 \in H^s_r(\mathbb{R})$ with $s > n + \frac{1}{2}$. For every $t \in \mathbb{R}$, the solution $u^\varepsilon(t)$ of \eqref{eq:BOneps} converges weakly in $L^2(\mathbb{R})$ to a function $ZD[u_0](t, \cdot)$, characterized as follows. Let $K_t(u_0)$ denote the set of critical values of the function
		$$
		y \in \mathbb{R} \mapsto y - (-1)^{n} (n+1) t u_0^n(y).
		$$
The set $K_t(u_0)$ is a compact subset of measure $0$. For every connected component $\Omega$ of $K_t(u_0)^c$, there exists a nonnegative integer $\ell $ such that, for every $x\in \Omega$, the equation
		$$
		y - (-1)^{n} (n+1) t u_0^n(y) = x
		$$
has $2\ell +1$ simple real solutions 
		$$
		y_0(t, x ) < y_1(t, x ) < \cdots < y_{2\ell }(t, x),
		$$
and the zero dispersion limit is given by 
\begin{align}\label{eq:zerolimitgeo}
ZD [u_0](t, x) = \sum_{k = 0}^{2 \ell} (-1)^k u_0 (y_k (t, x)).
\end{align}
\end{theorem}

\begin{remark}
The Lax–Levermore theory employed in \cite{miller2011zero} requires specific  constraints on the initial data. In contrast, based on the approach in \cite{Gerard2025small}, 
by approximating $u_0$ with rational functions possessing simple poles, we transform the formula into a manageable linear system.  This approach allows us to establish the signed sum \eqref{eq:zerolimitgeo} without imposing any essential constraints on the initial data profile, beyond the required regularity.
\end{remark}

\begin{remark}\label{remark-KdV}
The comparison with the KdV zero-dispersion problem and with the classification of traveling waves is instructive, and reflects a deeper difference between the two integrable structures.
On the traveling-wave side, the KdV classification follows from a direct ODE analysis, whereas for BO the nonlocal dispersion makes the problem subtler.
On the zero-dispersion side, as emphasized in \cite{Gerard2025small}, even the existence of a weak limit for every time and every initial datum in $L^\infty(\R)\cap L^2(\R)$ is still open for KdV. By contrast, the weak limit formula \eqref{eq:zerolimitgeo} obtained in this paper for the BO hierarchy is simple and more explicit than the corresponding formula found by Lax--Levermore for the KdV case. 
The existence of similar approach for KdV is an interesting open problem, and the methods developed here may provide further insights into it.
\end{remark}

\begin{remark}
Compared with the previous paper \cite{Gerard2025small} and lecture notes \cite{gerard2026lectures} of the first author, our proof of the weak convergence in Sections \ref{seczeroBo3} and \ref{seczeroBon} works directly with the resolvent equations (see Lemmas \ref{lemma:perturbedbo3} and \ref{lemma:perturbedbo_n}) and avoids treating the corresponding unbounded operators.
\end{remark}

\begin{remark}
While the statement of Theorem \ref{thm:zerolimit} covers the entire hierarchy $n \in \mathbb{N}$, our proof in Section \ref{seczeroBon} focuses on the case $n \geq 2$, which exhibits a significant algebraic shift from the $n=1$ case treated in \cite{Gerard2025small}. Specifically, the higher-order flows of the hierarchy ($n \geq 2$) necessitate a more general Vandermonde reduction, as developed in Proposition \ref{prop:ND_N}; in contrast, the standard \eqref{BO} case ($n=1$) collapses into a Cauchy-type determinant structure (see Remark \ref{remarkvander} for further technical details).
\end{remark}

\subsection{Structure of the paper}
This paper is structured as follows. In Section \ref{sectionlax}, we review the Lax pair structure of the Benjamin–Ono hierarchy. Section \ref{secproof} is devoted to the derivation of the general explicit formulas for solutions of the entire hierarchy. In Section \ref{sectiontw}, we provide a complete classification of traveling wave solutions. The study of the zero-dispersion limit begins in Section \ref{seczeroBo3} for the third-order case, which is subsequently generalized to all higher-order flows in Section \ref{seczeroBon}. Finally, the Appendix contains an algebraic lemma used in Sections~\ref{seczeroBo3}--\ref{seczeroBon}, which may be of independent interest in other contexts.

\section{The Lax operator}\label{sectionlax}
In this section, we present the Lax pair for the BO hierarchy, which will be used throughout the paper. Let $\sigma (L_u)$ be the spectrum of $L_u$.
For $\lambda \in \C \setminus \sigma (-L_u)$, we introduce the generating functional $\mathcal{H}_\lambda(u)$ (see \cite[Def. 2.14]{sun2021complete})
\begin{equation*}
	\mathcal{H}_\lambda(u) = \langle (L_u + \lambda)^{-1} \Pi u, \Pi u \rangle_{L^2}.
\end{equation*}
The $L^2$-gradient of this functional is given by $ \nabla \mathcal{H}_\lambda (u) = w_\lambda + \overline{w}_\lambda + |w_\lambda|^2$ where $w_\lambda := (L_{u} + \lambda)^{-1} \Pi u$, thus the BO hierarchy can be expressed by
\begin{equation}\label{BOhier}
	\partial_t u =\partial_x (\nabla \mathcal{H}_\lambda (u)) =  \partial_x \left( w_\lambda + \overline{w}_\lambda + |w_\lambda|^2 \right),
\end{equation}
where $ \partial_x (\nabla \mathcal{H}_\lambda (u))$ is the  Hamiltonian vector field of $\mathcal{H}_\lambda $ with
respect to the Gardner Poisson bracket. This bracket is defined for any two functionals $F, G : L^2(\mathbb{R}) \to \mathbb{C}$ with sufficiently regular $L^2$-gradients as:
\begin{equation*}
	\{F, G \} = \langle \partial_x \nabla F, \nabla G \rangle_{L^2}.
\end{equation*}
Under this Poisson structure, the generating function $\mathcal{H}_\lambda$ determines the time evolution of $u$ through the relation $\partial_t u = \{\mathcal{H}_\lambda, u\}$, which leads to the Hamiltonian vector field $\partial_x (\nabla \mathcal{H}_\lambda (u))$. 

The individual Hamiltonian of the BO hierarchy are recovered by expanding the $\mathcal{H}_\lambda(u)$ for spectral parameter $\lambda$,
\begin{equation}\label{eq:expanHlambda}
\mathcal{H}_\lambda(u) = \sum_{n=0}^\infty \frac{(-1)^{n} E_{n}(u)}{\lambda^{n+1}} = \frac{E_0}{\lambda} - \frac{E_1}{\lambda^2} + \frac{E_2}{\lambda^3} - \frac{E_3}{\lambda^4} + \dots,
\end{equation}
where the coefficients 
$E_n(u) := \langle L_u^{n} \Pi u, \Pi u \rangle_{L^2}$ constitute the infinite sequence of conservation laws for the hierarchy.
In particular, $E_0 (u)= \frac{1}{2}\|u\|_{L^2}^2$ denotes the momentum, and $E_1(u) =  \langle L_u \Pi u, \Pi u \rangle$ is the energy functional that generates the standard  \eqref{BO} equation, see also \eqref{eq:BOEner}. The next member, $E_2(u) = \langle L_u^2 \Pi u, \Pi u \rangle$, corresponds to the third-order BO flow. 
More generally, the $(n+1)$-th Hamiltonian flow in the hierarchy is
\begin{equation}\label{BOhiern}
	\partial_tu=\partial_x\bigl(\nabla E_n(u)\bigr).
\end{equation}

An important feature of the \eqref{BO} equation is that it admits a Lax pair (see \cite{fokas1983inverse, gerard2021integrability, wu2016simplicity}),
$$
\partial_t L_{u(t)} = \bigl[ B_{u(t)} , L_{u(t)}  \bigr],
$$
where $B_{u} := \ii(T_{|D| u} - T^2_{u}) $ is bounded and anti-selfadjoint. This structure extends to the entire BO hierarchy. Indeed, by \cite[Lemma. 2.20]{sun2021complete}, one has
$$
[D- T_{u}, T_{w_\lambda} T_{\overline{w_\lambda}} + T_{\overline{w_\lambda}} + T_{w_\lambda}] = T_{D(w_\lambda + \overline{w}_\lambda + |w_\lambda|^2)}.
$$
In other words,
\begin{equation}\label{eq:Lax_hierarchy}
\bigl[ B_u^{[\lambda]}, L_{u} \bigr] = - \ii T_{D\left( w_\lambda + \overline{w}_\lambda + |w_\lambda|^2 \right)}, \quad \text{with} \; B_u^{[\lambda]} :=  \ii \left( T_{w_\lambda} T_{\overline{w}_\lambda} + T_{\overline{w}_\lambda} + T_{w_\lambda} \right).
\end{equation}
It then follows from the definition of $L_{u(t)}$ and equation \eqref{BOhier} that
\begin{align*}
\partial_t L_{u(t)} 
= - T_{\partial_t u} 
= \bigl[ B_{u(t)}^{[\lambda]}, L_{u(t)} \bigr].
\end{align*}
We now derive the Lax pair for the individual Hamiltonian flows from the stationary identity
\eqref{eq:Lax_hierarchy}.  Let us expand $w_\lambda$,
$$
w_\lambda=(L_u+\lambda)^{-1}\Pi u
= \frac{1}{\lambda}\,\Pi u
- \frac{1}{\lambda^2}L_u\,\Pi u
+ \frac{1}{\lambda^3}L_u^2\,\Pi u
- \cdots
= \sum_{\delta=0}^\infty \frac{w_\lambda^{(\delta)}}{\lambda^{\delta+1}}, \, \text{with}\; w_\lambda^{(\delta)} = (-1)^{\delta} L_u^\delta \Pi u.$$
Substituting this series into the expression $B_u^{[\lambda]}$ in \eqref{eq:Lax_hierarchy}, we obtain
\begin{equation}\label{eq:Blambda}
B_u^{[\lambda]} = i \sum_{n=0}^\infty \frac{1}{\lambda^{n+1}} \Big[ \sum_{j=0}^{n-1} \big( T_{w_\lambda^{(j)}} T_{\overline{w_\lambda^{(n-1-j)}}} \big) + T_{w_\lambda^{(n)}} + T_{\overline{w_\lambda^{(n)}}} \Big].
\end{equation}
Here, the term involving the sum over $j$ is empty when $n=0$. We observe that $B_u^{[\lambda]}$ admits an expansion
$ B_u^{[\lambda]}
=
\sum_{n=0}^\infty \frac{(-1)^n}{\lambda^{n+1}}\,B_u^{(n)},$ with
\begin{equation}\label{eq:Bu_n_explicit}
	B_u^{(n)} = i \big( T_{L_u^n \Pi u} + T_{\overline{L_u^n \Pi u}} - \sum_{j=0}^{n-1} T_{L_u^j \Pi u} T_{\overline{L_u^{n-1-j} \Pi u}} \big).
\end{equation}
In particular, the case $n = 1$ is exactly the operator for standard \eqref{BO} equation,
$$
B_u^{(1)}=i\bigl(T_{L_u\Pi u}+T_{\overline{L_u\Pi u}}-T_{\Pi u}T_{\overline{\Pi u}}\bigr) =  i(T_{|D| u(t)} - T^2_{u(t)}),
$$
while $n=2$ corresponds to the third order equation of the BO hierarchy,
$$B_u^{(2)} = i \big( T_{L_u^2 \Pi u} + T_{\overline{L_u^2 \Pi u}} - T_{\Pi u} T_{\overline{L_u \Pi u}} - T_{L_u \Pi u} T_{\overline{\Pi u}} \big).$$

On the other hand, recall the expansion of $\mathcal{H}_\lambda$ in  \eqref{eq:expanHlambda}, 
we also have
$$
w_\lambda+\overline{w_\lambda}+|w_\lambda|^2
=
\nabla \mathcal H_\lambda(u)
=
\sum_{n=0}^\infty \frac{(-1)^n}{\lambda^{n+1}}\,\nabla E_n(u).
$$
Therefore,
$$
-i\,T_{D(|w_\lambda|^2+w_\lambda+\overline{w_\lambda})}
=
\sum_{n=0}^\infty \frac{(-1)^n}{\lambda^{n+1}}
\bigl(-i\,T_{D(\nabla E_n(u))}\bigr).
$$
Combining with the identity \eqref{eq:Lax_hierarchy}, and matching the coefficient of $\lambda^{-(n+1)}$, we obtain
$$
[B_u^{(n)},L_u]
=
-i\,T_{D(\nabla E_n(u))}.
$$
Hence, the corresponding Lax pair for the $(n+1)$-th order flow is given by
$$
\partial_tL_{u(t)}
=
-T_{\partial_tu}
=
-i\,T_{D(\nabla E_n(u))}
=
[B_{u(t)}^{(n)},L_{u(t)}].
$$

As seen from \eqref{eq:Bu_n_explicit}, the operator $B_u^{(n)}$ is a finite sum of linear and quadratic terms involving Toeplitz operators. Although these expressions become complex for larger 
$n$, their specific structure is sufficient for our analysis, and we shall not require more detailed representations in our proofs.
The next lemma gives some basic properties of $B_u^{(n)}$. 
\begin{lemma}[Boundedness of the operator $B_u^{(n)}$]\label{lemma:boundedB}
Let $n \in \mathbb{N}$ and $u \in H^s(\mathbb{R})$ with $s > n + \frac{1}{2}$. Then the operator $B_u^{(n)}$ is antiself-adjoint and bounded on $L^2_+(\mathbb{R})$.
\end{lemma}
\begin{proof}
The anti-selfadjointness is immediate from the definition. Indeed 
$$ (B_u^{(n)})^* = -i \big( T_{\overline{L_u^n \Pi u}} + T_{L_u^n \Pi u} - \sum_{j=0}^{n-1} (T_{L_u^j \Pi u} T_{\overline{L_u^{n-1-j} \Pi u}})^* \big),$$
and 
$$
(T_{L_u^j\Pi u}T_{\overline{L_u^{n-1-j}\Pi u}})^*
=T_{L_u^{n-1-j}\Pi u}T_{\overline{L_u^j\Pi u}},$$
so the quadratic sum is invariant under the reindexing. Hence, $(B_u^{(n)})^* = -B_u^{(n)}$.
	
For the boundedness, it suffices to check that every symbol $L_u^k \Pi u$ belongs to $L^\infty(\mathbb{R})$ for $0 \le k \le n$, since for every bounded symbol $b$ one has $\|T_u\| \leq \|u\|_{L^\infty}$. Recall the definition of the Lax operator $L_u := D - T_u$. We first show that
\begin{equation}\label{eq:Luinduction}
L_u^k\Pi u \in H_+^{s-k}(\R), \qquad 0\le k\le n.
\end{equation}
The case $k=0$ is immediate, since $\Pi u\in H_+^s(\R)$. Now let $0\le k\le n-1$ and assume that $
L_u^k\Pi u \in H_+^{s-k}(\R). $ Since $u\in H^s(\R)\subset H^{s-k}(\R)$, multiplication by $u$ is bounded on $H^{s-k}(\R)$, and therefore, using also the boundedness of $\Pi$ on $H^{s-k}(\R)$, the Toeplitz operator
$
T_u=\Pi(u\cdot):H_+^{s-k}(\R)\to H_+^{s-k}(\R)
$
is bounded. On the other hand,
$D:H_+^{s-k}(\R)\to H_+^{s-k-1}(\R).$
It follows that
$$
L_u=D-T_u:H_+^{s-k}(\R)\to H_+^{s-k-1}(\R),
$$
and therefore
$
L_u^{k+1}\Pi u \in H_+^{s-k-1}(\R).$
By induction, we obtain \eqref{eq:Luinduction}.
In particular, by the hypothesis $s > n + 1/2$ and the Sobolev embedding ($H^\sigma(\mathbb{R}) \hookrightarrow L^\infty(\mathbb{R})$ for $\sigma > 1/2$), we obtain
$$
L_u^k\Pi u\in H^{s-k}(\R)\hookrightarrow L^\infty(\R), \qquad 0\le k\le n.
$$
Hence, all symbols and their conjugates in the expression for $B_u^{(n)}$ belong to $L^\infty(\mathbb{R})$, which concludes the boundedness of $B_u^{(n)}$.
\end{proof}

As a consequence, if we define the family of bounded operators
$\{U(t)\}_{t\in \mathbb{R}}$ on $L^2_+(\mathbb{R})$ by the linear ODE 
$$
\frac{d}{dt}U(t)=B_{u(t)}^{(n)}U(t),\qquad U(0)=\mathrm{Id},$$
then $U(t)$ is unitary and 
$$L_{u(t)} = U(t) L_{u(0)} U(t)^*.$$
The identity shows that the operator $L_{u(t)} $ evolves by unitary conjugation, which implies that its spectrum is preserved. This isospectral property is a fundamental feature of the hierarchy and will be used throughout our proofs. An analogous property holds for the standard \eqref{BO} equation (see \cite[Prop. 4.4]{gerard2026lectures} for a detailed proof).

\section{Proof of the explicit formulae on the line}\label{secproof}
This section is devoted to the proof of Theorem \ref{mainBOhier}.  For the sake of clarity, we proceed in two steps. We first prove the explicit formula for the third-order flow, which clarifies the main ideas. Subsequently, we return to the general $(n+1)$-st flow and show that the same argument applies to higher-order flow of the BO hierarchy, thereby yielding  the explicit formula \eqref{BOhierexpfor}.

Following \eqref{BOhiern}, the third-order equation of the BO hierarchy is given by
\begin{equation}\label{BOhier3}
	\partial_t u = \partial_x (\nabla E_2(u)).
\end{equation}
where the Hamiltonian is defined as
$
E_2(u) = \langle (D - T_u)^{2} \Pi u, \Pi u \rangle_{L^2}.$ Alternatively, the equation \eqref{BOhier3} can be explicitly written as
\begin{equation}\label{BO3}
		\partial_t u
		+ \partial_x^3 u
		+ \frac{3}{2} \partial_x(u |D| u)
		+ \frac{3}{2} |D|(u \partial_x u)
		- \partial_x (u^3) = 0.
\end{equation}
We refer the reader to \cite[Prop. A.2]{gassot2021third} for a detailed derivation of this form. The next theorem is precisely Theorem \ref{mainBOhier} specialized to this model case.
\begin{theorem}\label{mainBO3}
	Let $u \in C(\mathbb{R}, H^s_r(\mathbb{R}))$, $s > \frac{5}{2}$, be the solution of \eqref{BO3} on the line $ \mathbb{R} $ with $ u(0) = u_0 $.  
	Then $u (t, x)= \Pi u(t, x) + \overline{\Pi u(t, x)}, \ x\in \R  $, with
	$$
	\forall\; z \in \mathbb{C}_+, \quad \Pi u(t, z) = \frac{1}{2 \ii \pi} \, I_+ \left[ \left( X^* - 3t L_{u_0}^2 - z \mathrm{Id} \right)^{-1} \Pi u_0 \right].
	$$
\end{theorem}

We first give a commutator identity between $X^*$ and the $B$ operator, which is the key point for the evolution of the conjugated operator $U(t)^*X^*U(t)$. We state it first in the third-order case, and at the same time give the general formula that will be needed later.
\begin{proposition}\label{PropXstatBn}
For every $u\in H_r^{s}(\R)$, $s > \frac{5}{2}$,
\begin{equation}\label{eq:XstarB2}
 [X^*, B_u^{(2)}] f = \frac{1}{2\pi}\Bigl(\langle f|L_u\Pi u\rangle\Pi u
+ \langle f|\Pi u\rangle L_u\Pi u
- I_+(f)L_u^2\Pi u\Bigr).
\end{equation}
More general, given $n\in \N$, for every $u\in H_r^{s}(\R)$, $s > n+ \frac{1}{2}$,
\begin{equation}\label{eq:XstarBn}
 [X^*, B_u^{(n)}] f 
= \frac{1}{2\pi}  \left(  \sum_{j=0}^{n-1} \langle f, L_u^{n-1-j} \Pi u \rangle L_u^j \Pi u - I_+(f) L_u^n \Pi u \right).
\end{equation}
\end{proposition}
\begin{proof}
We begin by characterizing the commutator $ [X^*, B_u^{[\lambda]}]$. By the expression $$B_u^{[\lambda]} = \ii \left( T_{w_\lambda} T_{\overline{w_\lambda}} + T_{\overline{w_\lambda}} +  T_{w_\lambda} \right),$$
we can split the commutator into three parts 
\begin{align*}
	[X^*,B_u^{[\lambda]}]
	= 
	 \ii\bigl[X^*,T_{w_\lambda}\bigr]
	+ \ii\bigl[X^*,T_{\overline{w_\lambda}}\bigr] + \ii\bigl[X^*,T_{w_\lambda}\overline{T_{w_\lambda}}\bigr].
\end{align*}
We use the elementary identity $[X^*,T_b] =  \frac{\ii}{2\pi}I_+(\cdot)\,\Pi b$ (see  \cite{gerard2021integrability, sun2021complete})), to compute each term separately. Since $w_\lambda\in L_+^2(\R)$, we have
$$
[X^*,T_{w_\lambda}]f=\frac{\ii}{2\pi}\,I_+(f)\, \Pi w_\lambda =\frac{\ii}{2\pi}I_+(f)w_\lambda,
\qquad
[X^*,T_{\overline{w_\lambda}}]f=0.
$$
For the quadratic term, the Leibniz rule gives
\begin{align*}
	[X^*,T_{w_\lambda}\,T_{\overline{w_\lambda}}]f
	= [X^*,T_{w_\lambda}]\,T_{\overline{w_\lambda}}f
	= \frac{\ii}{2\pi} I_+(T_{\overline{w_\lambda}}f)\, \Pi w_\lambda
	= \frac{\ii}{2\pi} I_+(\Pi(\overline{w_\lambda} f))\,w_\lambda
	= \frac{\ii}{2\pi}\langle f|w_\lambda\rangle w_\lambda,
\end{align*}
where we used the fact that $ I_+( \Pi (g)) = \langle g| {\bf{1}} \rangle $ when $g \in H_+$.
Substituting all terms into the main formula, we have
$$
[X^*,B_u^{[\lambda]}]f
= -\frac{1}{2\pi}\Bigl(\langle f|w_\lambda\rangle w_\lambda
+ I_+(f)\,w_\lambda\Bigr).
$$
Now, we expand
$w_\lambda
= \sum_{\delta=0}^\infty \frac{1}{\lambda^{\delta+1}}w_\lambda^{(\delta)}
$ with $w_\lambda^{(\delta)} = (-1)^{\delta} L_u^\delta \Pi u $, and compare this with the expression of $B_u^{[\lambda]}$,
\begin{equation*}
B_u^{[\lambda]}
= \frac{1}{\lambda}B_u^{(0)}
- \frac{1}{\lambda^2}B_u^{(1)}
+ \frac{1}{\lambda^3}B_u^{(2)}
- \frac{1}{\lambda^4}B_u^{(3)} + \cdots
= \sum_{n=0}^\infty \frac{(-1)^{n}}{\lambda^{n+1}}B_u^{(n)}.	
\end{equation*}
By matching the powers of $\lambda$ on both sides of the commutator identity, we can extract the commutator $ [X^*, B_u^{(2)}]$ and $ [X^*, B_u^{(n)}]$ to get the desired formulas \eqref{eq:XstarB2} and \eqref{eq:XstarBn}. This completes the derivation.
\end{proof}

The next proposition rewrites the third-order commutator, which encodes how the operator $B_{u}^{(2)}$ deforms the spectral structure associated with $ L_{u}$.
\begin{proposition}\label{Prop1a}
For every $u\in H_r^{s}(\R)$, $s > \frac{5}{2}$,
$$ \left[ X^*, B_{u}^{(2)} \right]  = -3 L_{u}^2 - \ii [X^*, L_{u}^3].$$
\end{proposition}
\begin{proof}
As the expression of $[X^*,B_u^{(2)}]$ is given in Proposition \ref{PropXstatBn}, it remains to treat the commutator $[X^*,L_u^3]$. 
By the Leibniz rule for commutators, we have
\begin{align*}
	[X^*,L_u^3]
	&= [X^*,L_u]\,L_u^2 + L_u\,[X^*,L_u]\,L_u + L_u^2\,[X^*,L_u].
\end{align*}
Recall that $[X^*, D] = \mathrm{Id}$ and $[X^*,T_u] =  \frac{\ii}{2\pi}I_+(\cdot)\,\Pi u$, where the proof can be found in \cite{gerard2021integrability, sun2021complete}. Hence,
$$
[X^*,L_u] = \ii \mathrm{Id} - [X^*,T_u]
= \ii \mathrm{Id} - \frac{\ii}{2\pi}I_+(\cdot)\,\Pi u.
$$
With this expression of $[X^*,L_u]$, we substitute back into our earlier formula for $[X^*,L_u^3]$, and perform the algebraic computation explicitly. This yields: 
\begin{align*}
[X^*,L_u^3]f
& = 3 \ii \,L_u^3f  - \frac{\ii}{2\pi} \Bigl( I_+(L_u^2 f)\,\Pi u + L_u I_+(L_uf)\,\Pi u +  L_u^2 I_+(f)\,\Pi u
\Bigr)
\end{align*}
By the definition of $I_+$ and notice that both $L_u$ and $T_u$ are selfadjoint, we have 
\begin{eqnarray*}
I_+(L_u f)= - I_+(T_u f) = - \langle T_u f|  {\bf{1}} \rangle   = - \langle f| \Pi u\rangle,  \\
I_+(L_u^2 f) =- \langle L_u f|\Pi u\rangle = - \langle f|L_u\Pi u\rangle.
\end{eqnarray*}
Hence,
\begin{align*}
[X^*,L_u^3]f= 3\ii \,L_u^3f
+ \frac{\ii}{2\pi}\Bigl(\langle f|L_u\Pi u\rangle\Pi u
+ \langle f|\Pi u\rangle L_u\Pi u
- I_+(f)L_u^2\Pi u\Bigr).
\end{align*}
Comparing this formula with identity \eqref{eq:XstarB2} yields the Proposition \ref{Prop1a}.
\end{proof}

\begin{proposition}\label{Prop2}
For the operator $U(t)$ of the third equation of the Benjamin–Ono equation, we have 	
\begin{align*}
U^*(t)  \chi_\varepsilon = e^{i t L_{u_0}^3} \chi_\varepsilon + o(1), \qquad
U^*(t) \Pi u = e^{i t L_{u_0}^3} \Pi u_0.
\end{align*}
\end{proposition}
\begin{proof}
By the evolution equation $\frac{d}{dt}U(t) = B_{u}^{(2)} U(t)$ for the unitary operator $ U(t)$, and the fact that the operator $B_{u}^{(2)}$ is anti-selfadjoint, we have
\begin{equation}\label{eq:devU*}
\frac{d}{dt} U^*(t) = - U^*(t) B_{u}^{(2)}. 
\end{equation}
Moreover, for $f\in L_+^2(\R)$, one has
$$
T_f\chi_\varepsilon=f+o(1),
\qquad
T_{\overline f}\chi_\varepsilon=o(1)
\qquad (\varepsilon\to0^+).$$
Applying this to $ B_{u}^{[\lambda]}$, we obtain 
$$
B_u^{[\lambda]} \chi_\varepsilon = \ii \left( T_{w_\lambda} T_{\overline{w_\lambda}} + T_{\overline{w_\lambda}} +  T_{w_\lambda} \right) \chi_\varepsilon 
= \ii (w_\lambda + o(1)).
$$
We recall the expansions of $w_\lambda$ and $B_u^{[\lambda]}$ in powers of $\lambda^{-1}$,
\begin{align*}
w_\lambda
= \frac{1}{\lambda}\,\Pi u
- \frac{1}{\lambda^2}L_u\,\Pi u
+ \frac{1}{\lambda^3}L_u^2\,\Pi u \cdots,\quad 
B_u^{[\lambda]}
 = \frac{1}{\lambda}B_u^{(0)}
- \frac{1}{\lambda^2}B_u^{(1)}
+ \frac{1}{\lambda^3}B_u^{(2)}\cdots
\end{align*}
By matching the powers,  we get
\begin{equation}\label{eq:Buchie}
B_u^{(2)}\chi_\varepsilon = \ii L_u^2\,\Pi u + o(1).
\end{equation}
On the other hand,
\begin{equation}\label{eq:Luchie}
	L_u\chi_\varepsilon=(D-T_u)\chi_\varepsilon=o(1)-\Pi(u\chi_\varepsilon)=-\Pi u+o(1),
\end{equation}
so $L_u^3\chi_\varepsilon=-L_u^2\Pi u+o(1)$. 
Combining this with \eqref{eq:Buchie}, we find
$$
B_u^{(2)}\chi_\varepsilon=-iL_u^3\chi_\varepsilon+o(1).
$$
Substituting this relation into the equation \eqref{eq:devU*}, we find 
\begin{equation*}
	\frac{d}{dt} U^*(t)\chi_\varepsilon = - U^*(t) B_{u}^{(2)} \chi_\varepsilon = i U^*(t) L_u^3\chi_\varepsilon + o(1) = iL_{u_0}^3U^*(t)\chi_\varepsilon+o(1).
\end{equation*}
Solving this ODE yields the first identity.\\
For the second identity, using the relation \eqref{eq:Luchie} and identity $U^*(t) L_{u(t)} U(t) = L_{u_0}$, we have 
\begin{align*}	
U^*(t) \Pi u
& = U^*(t)(- L_u \chi_\varepsilon + o(1)) 
= - L_{u_0} U^* (t) \chi_\varepsilon + o(1).
\end{align*} 
Therefore, 
$$
U^*(t) \Pi u =  - L_{u_0} e^{i t L_{u_0}^3} \chi_\varepsilon + o(1).	
$$
Using again \eqref{eq:Luchie} and the fact that $L_{u_0}$ commutes with $e^{itL_{u_0}^3}$, we obtain $U^*(t)\Pi u(t)=e^{itL_{u_0}^3}\Pi u_0.$
\end{proof}

We can now prove the explicit formula in the Theorem \ref{mainBO3} for the third-order flow.
\begin{proof}[Proof of Theorem \ref{mainBO3}]
For every \( z \in \mathbb{C}_+ \), by the inverse Fourier transform and Plancherel Theorem, we have
\begin{align*}
	\Pi u(t, z) 
	&= \frac{1}{2\ii  \pi} \int_0^{\infty} e^{i z \xi} \, \widehat{\Pi u}(t, \xi) \, d\xi \\
	& =  \frac{1}{2\ii \pi} \int_0^{\infty} e^{i z \xi} \lim_{\varepsilon \to 0} \int_{\mathbb{R}} e^{-i \xi x} \frac{\Pi u(t,  x)}{1 + \ii \varepsilon x} \, dx  \\
	& =  \lim_{\varepsilon \to 0^+} \frac{1}{2\ii  \pi} \int_0^{\infty} e^{i z \xi} \left\langle e^{- \ii \xi X^*} \Pi u \mid \chi_\varepsilon \right\rangle \, d\xi
	= \lim_{\varepsilon \to 0^+} \frac{1}{2 \ii  \pi} \left\langle \left( X^* - z \, \mathrm{Id} \right)^{-1} \Pi u \mid \chi_\varepsilon \right\rangle.
\end{align*}
Since $U(t)$ is a unitary operator and it preserves the inner product, we infer that
\begin{equation}\label{eq:Piutz}
	\begin{aligned}
		\Pi(u)(t,z) 
		& = \lim_{\varepsilon \to 0 } \frac{1}{2 \ii \pi} \big\langle \big( U^*(t)  X^* - z  {{{\rm{Id}}}}  \big)^{-1} \Pi(u_0)\, \vert \, { U^*(t)  \chi_\varepsilon}\big\rangle\\	
		& = \lim_{\varepsilon \to 0 } \frac{1}{2 \ii \pi} \big\langle \big( U^*(t) X^* U(t) - z  {{{\rm{Id}}}}  \big)^{-1} U^*(t) \Pi(u_0)\, \vert \, { U^*(t)  \chi_\varepsilon}\big\rangle.
	\end{aligned}
\end{equation}

Now, it remains to obtain the expression of the quantity $U^*(t) X^* U(t)$. 
By the evolution equations for the operator $ U(t)$ and its adjoint operator $ U^*(t)$
\begin{align*}
\frac{d}{dt}U(t) = B_{u}^{(2)} U(t), \qquad
\frac{d}{dt} U^*(t) = - U^*(t) B_{u}^{(2)},
\end{align*}
we compute the time derivative of the conjugated operator \( U^*(t) X^* U(t) \),
\begin{align*}
		\frac{d}{dt} U^*(t) X^* U(t) 
		= U^*(t) \left[ X^*, B_{u}^{(2)} \right] U(t)
		= U^*(t) \left( -3 L_{u(t)}^2 - i [X^*, L_{u(t)}^3] \right) U(t),
\end{align*}
where we used the known commutator identities (from Proposition \ref{Prop1a}). Recall the key transform identity $U^*(t) L_{u(t)} U(t) = L_{u_0}$, we infer
$$
\frac{d}{dt} \left( U^*(t) X^* U(t) \right) = -3 L_{u_0}^2 - i \left[ U^*(t) X^* U(t), L_{u_0}^3 \right].
$$
Using the integrating factors $e^{-i t L_{u_0}^3}$ and $e^{i t L_{u_0}^3}$, we solve this ODE to get 
$$
e^{-i t L_{u_0}^3} \, U^*(t) X^* U(t) \, e^{i t L_{u_0}^3}
= -3t L_{u_0}^2 + U(0)^* X^* U(0).$$
Since \( U(0) = \mathrm{Id} \), the last term becomes $X^* $. Thus, 
\begin{align}\label{eqUXU}
	U^*(t) X^* U(t) = -3t L_{u_0}^2 + e^{i t L_{u_0}^3} X^* e^{-i t L_{u_0}^3}.
\end{align}
We now insert \eqref{eqUXU} into \eqref{eq:Piutz} and use Proposition \ref{Prop2},
\begin{equation*}
	\begin{aligned}
		\Pi(u)(t, z) 
		&= \lim_{\varepsilon \to 0} \frac{1}{2 \ii \pi} 
		\left\langle \left( -3t L_{u_0}^2 + e^{i t L_{u_0}^3} X^* e^{- \ii t L_{u_0}^3} - z \,\mathrm{Id} \right)^{-1} e^{i t L_{u_0}^3} \Pi(u_0) 
		\,\middle|\, e^{i t L_{u_0}^3} \chi_\varepsilon \right\rangle \\
		& = \lim_{\varepsilon \to 0} \frac{1}{2 \ii \pi} 
		\left\langle \left( X^*  - 3t L_{u_0}^2- z \,\mathrm{Id} \right)^{-1}  \Pi(u_0) 
		\,\middle|\,  \chi_\varepsilon \right\rangle.
	\end{aligned}
\end{equation*}
By the definition of $I_+$, this is exactly
$$\forall z \in \mathbb{C}_+, \quad \Pi u(t, z) = \frac{1}{2 \ii \pi} \, I_+ \left[ \left( X^* - 3t L_{u_0}^2 - z \mathrm{Id} \right)^{-1} \Pi u_0 \right].$$
The proof of Theorem \ref{mainBO3} is complete. 
\end{proof}	

We now return to the full hierarchy. Once the third-order case is understood, the passage to general $n$ is direct. The next proposition is the exact analogue of Proposition \ref{Prop1a}.
\begin{proposition}\label{Prop1n}
Given $n\in \N$, for every $u\in H_r^{s}(\R)$, $s > n +\frac{1}{2}$,
$$ \left[ X^*, B_{u }^{(n)} \right]  = - (n+1) L_{u}^n - \ii [X^*, L_{u}^{n+1}].$$
\end{proposition}
\begin{proof}
We start with the computation of the commutator $[X^*, L_u^{n+1}]$. 
Applying the general Leibniz rule, we get
$$
[X^*, L_u^{n+1}] = \sum_{j=0}^n L_u^j [X^*, L_u] L_u^{n-j}.$$
As in the proof of Proposition \ref{Prop1a}, we substitute the elementary identity $ [X^*, L_u] f = \ii f - \frac{i}{2\pi} I_+(f) \Pi u$ into the sum above,
	\begin{align*}
		[X^*, L_u^{n+1}] f 
		&= \sum_{j=0}^n L_u^j \Big( \ii L_u^{n-j} f - \frac{\ii}{2\pi} I_+(L_u^{n-j} f) \Pi u \Big)\\
		&=(n+1) i L_u^{n} f - \frac{\ii}{2\pi} \sum_{j=0}^n L_u^j \left( I_+(L_u^{n-j} f) \Pi u \right).
	\end{align*}
For $0\leq j\leq n-1$,
	$$
	I_+(L_u^{n-j} f) = I_+\left( L_u (L_u^{n-j-1} f) \right) = -\langle L_u^{n-1-j} f| \Pi u \rangle = - \langle f| L_u^{n-1-j} \Pi u \rangle.
	$$
and the term $j=n$ is simply $I_+(f)$. Therefore,
	$$[X^*, L_u^{n+1}] f  = (n+1) \ii L_u^{n} f + \frac{\ii}{2\pi} \Bigl( 
	\sum_{j=0}^{n-1} \langle f| L_u^{n-1-j} \Pi u \rangle L_u^j \Pi u - I_+(f) L_u^n \Pi u \Bigr).
	$$
Recall the  general expression of $[X^*, B_u^{(n)}] f$ in Proposition \ref{PropXstatBn}, we finally have
\begin{align*}
[X^*, B_u^{(n)}] f 
 = \frac{1}{2\pi} \Big(  \sum_{j=0}^{n-1} \langle f| L_u^{n-1-j} \Pi u \rangle L_u^j \Pi u - I_+(f) L_u^n \Pi u \Big) 
= -  \ii [X^*, L_u^{n+1}] f - (n+1) L_u^{n} f.
\end{align*}
The proof of Proposition \ref{Prop1n} is complete.
\end{proof}

With this identity in hand, using the same strategy as in the proof of Theorem \ref{mainBO3}, we can derive the explicit formulae for all BO hierarchy. 
We include the argument for completeness, in order to make clear where the factor $(n+1)$ and the power $L_{u_0}^{n+1}$ enter.
\begin{proof}[Proof of Theorem \ref{mainBOhier}]
As in the proof of Theorem \ref{mainBO3}, we start from
\begin{equation}\label{eq:Piu}
\Pi u (t, z) = \frac{1}{2\pi \ii} \lim_{\varepsilon \to 0} \left\langle \left( U(t)^* X^* U(t) - z \mathrm{Id} \right)^{-1} U(t)^* \Pi u (t, z) \middle| U(t)^* \chi_\varepsilon \right\rangle.
\end{equation}
Let us compute again the time derivative of the conjugated operator $U^*(t) X^* U(t)$, and use Proposition \ref{Prop1n},
\begin{align*}
	\frac{d}{dt} U(t)^* X^* U(t) 
	& = U(t)^* [X^*, B^{(n)}_u ]U(t)\\
	& = U(t)^* \left( - \ii [X^*, L_u^{n+1}]  - (n+1) L_u^{n}   \right) U(t),
\end{align*}
from which we have
\begin{equation*}
	U^*(t) X^* U(t) =
			- (n+1)  t L_{u_0}^n + e^{  i t L_{u_0}^{n+1}} X^* e^{-i t L_{u_0}^{n+1}}.
\end{equation*}	
In the same way as in Proposition \ref{Prop2}, comparing the coefficient of $\lambda^{-n-1}$ in
$$
B_u^{[\lambda]}\chi_\varepsilon= \ii w_\lambda+o(1)
$$
gives
\begin{align}\label{eq:Uchin}
 U^*(t) \Pi u =  e^{itL_{u_0}^{n+1} } \Pi u_0 ,  \qquad
 U^*(t)  \chi_\varepsilon = e^{itL_{u_0}^{n+1} } \chi_\varepsilon + o(1).
\end{align}
Substitute the above identities into \eqref{eq:Piu}, we obtain
\begin{align*}
	& \Pi u (t, z) = \\
	& \frac{1}{2\pi \ii} \lim_{\varepsilon \to 0} \left\langle \left( - (n+1)t L_{u_0}^n + e^{\ii tL_{u_0}^{n+1}} X^* e^{- \ii tL_{u_0}^{n+1}} - z I d \right)^{-1} e^{itL_{u_0}^{n+1} } \Pi u_0 \middle| e^{itL_{u_0}^{n+1} } \chi_\varepsilon \right\rangle\\
	& = \frac{1}{2\pi \ii} I_+ \left[ \left( X^* - (n+1)t L_{u_0}^n - z I d \right)^{-1} \Pi u_0\right].
\end{align*}
This is exactly the formulae \eqref{BOhierexpfor} and the proof of Theorem \ref{mainBOhier} is complete.
\end{proof}

\begin{remark}
It is worth noting that, unlike the derivation in \cite{gerard2023explicit}, our proof utilizes the asymptotic expansion of $B_u^{[\lambda]}$ directly, which bypasses the need for an explicit representation of the operator $B_u^{(n)}$.	
\end{remark}

We conclude this section with a remark that will be used repeatedly in Sections \ref{seczeroBo3} and \ref{seczeroBon}. The explicit formulas above involve resolvents of unbounded operators such as $X^*-3tL_{u_0}^2$ and $X^*-(n+1)tL_{u_0}^n$. Before applying these formulas in the zero-dispersion analysis, one therefore has to know that the corresponding resolvents are well defined. Here, we give the definition of the maximally dissipative operators \cite{reed2003methods}.

\begin{definition}[Maximally dissipative operator]
A linear operator $A$ on a Hilbert space $H$ is called maximally dissipative if, for every $\lambda >0$, the operator $\lambda \text{Id} - A: \text{Dom}(A) \to H$ is onto and if, all $f \in \text{Dom}(A)$, 
$$\text{Re}\langle Af, f \rangle \leq 0.$$
If $A$ is maximally dissipative and $\text{Im}(z) > 0$, then
$$ \|(A - z I)^{-1}\| \leq \frac{1}{\text{Im}(z)}. $$
\end{definition}

\begin{remark}\label{remark:well-def}
In the derivation of the explicit formula for the third-order flow, we obtained 
$$
e^{-i t L_{u_0}^3} \, U^*(t) X^* U(t) \, e^{ \ii t L_{u_0}^3}
= X^* -3t L_{u_0}^2,$$
and, more generally,
\begin{equation} \label{eq:sanwich_n}
e^{- \ii tL_{u_0}^{n+1}}U^*(t)X^*U(t)e^{ \ii tL_{u_0}^{n+1}}=X^*-(n+1)tL_{u_0}^n.
\end{equation}
Setting $A_0 := - \ii X^*$, which is known to be maximally dissipative,  we show that the operator $A_n := -\ii(X^* - (n+1)t L_{u_0}^n)$ is also maximally dissipative. Since both $U(t)$ and $e^{it L_{u_0}^{n+1}}$ are the unitary, their product $V(t) := U(t) e^{ \ii t L_{u_0}^{n+1}}$ is unitary on $L^2_+(\mathbb{R})$. From \eqref{eq:sanwich_n}, it follows that
$$A_n = V(t)^* (- \ii X^*) V(t) = V(t)^* A_0 V(t).$$
To verify the dissipativity of $A_n$, let $f \in \text{Dom}(A_n)$, then $V(t) f \in \text{Dom}(A_0)$. 
Using the unitarity of $V(t)$ ($V^*V = \text{Id}$), we have
\begin{align*}
	\operatorname{Re} \langle A_n f, f \rangle = \operatorname{Re} \langle V(t)^* A_0 V(t) f, f \rangle 
	= \operatorname{Re} \langle A_0 (V(t)f), V(t)f \rangle \leq 0,
\end{align*}
as $A_0$ is dissipative. 
To prove maximality, using again the unitarity of $V(t)$, for any $\lambda >0$, we have
\begin{align*}
	\lambda \text{Id} - A_n = \lambda V(t)^* V(t) - V(t)^* A_0 V(t) 
	= V(t)^* \left( \lambda \text{Id} - A_0 \right) V(t).
\end{align*}
Since $V(t)$ is bijective and the operator $(\lambda \text{Id} - A_0)$ is surjective, it follows that $\lambda \text{Id} - A_n $ is  surjective. 
Thus, $A_n$ is maximally dissipative. Therefore, for every $z\in\C_+$,
$$
\bigl\|(X^*-(n+1)tL_{u_0}^n-z\mathrm{Id})^{-1}\bigr\|\leq\frac1{\operatorname{Im}z},
$$
and similarly in the third-order case with $X^*-3tL_{u_0}^2$.

Consequently, the resolvents appearing in Theorem \ref{mainBO3} and in Theorem \ref{mainBOhier} are well defined for every $z\in\C_+$.
\end{remark}
 
\section{Classification of traveling waves}\label{sectiontw} 
This section is devoted to the proof of Theorem~\ref{thm:classification}.
We begin with a uniform estimate for the explicit formula, which will be used  subsequently. 
\begin{lemma}\label{L2boundexplicit}
Let $f \in L^2_+(\mathbb{R})$, $t \in \mathbb{R}$, and $n \in \mathbb{N}$. We define the function $\Omega_t^{(n)} f$ associated with the $(n+1)$-th order flow as:
\begin{equation*}
	\forall z \in \mathbb{C}_+,\quad
	\Omega_t^{(n)} f(z) := \frac{1}{2 \ii \pi} I_+ \left( \left( X^* - (n+1)t L_{u_0}^n - z \mathrm{Id} \right)^{-1} f \right).
\end{equation*}
Then $\Omega_t^{(n)} f$ belongs to $L^2_+(\R)$, and $\| \Omega_t^{(n)} f\|_{L^2}\leq \| f\|_{L^2}$. In particular, for every $z\in\C_+$, we have
$$ |\Omega_t^{(n)} f(z)|\leq \frac{\| f\|_{L^2}}{2\sqrt{\pi \mathrm{Im}z}}. $$
\end{lemma}

\begin{proof}
Assume first that $f \in H^s_+(\mathbb{R})$ with $s > n + \frac{1}{2}$. By formula \eqref{eq:sanwich_n} in Remark~\ref{remark:well-def}, we have, for every $z\in\C_+$,

$$
\Omega_t^{(n)}f(z)
= 
\frac1{2\ii \pi }
I_+ \big(e^{-\ii tL_{u_0}^{n+1}}U(t)^*(X^*-z)^{-1} U(t)e^{\ii tL_{u_0}^{n+1}}  f\big) ,
$$
where $U(t)$ is the unitary solution of the linear ODE $U'(t) = B_{u}^{(n)} U(t)$, $U(0) = \mathrm{Id}$. 
Moreover, by \eqref{eq:Uchin}, we have
$
e^{- \ii tL_{u_0}^{n+1}}U(t)^*\chi_\varepsilon
=
\chi_\varepsilon+o(1)$. It follows that, for every $h\in L^2_+(\R)$,
$$
I_+\bigl(e^{-\ii tL_{u_0}^{n+1}}U(t)^* h\bigr)=I_+(h).
$$
Therefore, 
\begin{align*}
\Omega_t^{(n)}f (z) = 
\frac1{2 \ii \pi}
I_+\big((X^*-z)^{-1} U(t)e^{ \ii tL_{u_0}^{n+1}}	 f\big) =
(U(t)e^{\ii tL_{u_0}^{n+1}}) f (z).
\end{align*}
Then the $L^2_+$ case
follows by a density argument and passage to the limit. 
By the fact that $U(t)$ and $e^{\ii tL_{u_0}^{n+1}}$ are unitary on $L^2_+(\R)$, $\Omega_t^{(n)}$ is unitary, 
 and we deduce that
$$
\Omega_t^{(n)}f\in L^2_+(\R)
\qquad\text{and}\qquad
\|\Omega_t^{(n)}f\|_{L^2} \le \|f\|_{L^2}.
$$
Now set
$g:=\Omega_t^{(n)}f.$
Since $g\in L^2_+(\R)$, its holomorphic extension to $\C_+$ is given by
$$
g(z)=\frac{1}{2\pi}\int_0^\infty e^{ \ii z\xi}\widehat g(\xi)\,d\xi,
\qquad z\in\C_+.
$$
By the Cauchy--Schwarz inequality and Plancherel identity, we obtain
\begin{align*}
|g(z)|
\le 
\frac{1}{2\pi \sqrt{2 \mathrm{Im}  z}}
\|\widehat g\|_{L^2(0,\infty)} = \frac{\|g\|_{L^2}}{2\sqrt{\pi\,\mathrm{Im} z}}
\le
\frac{\|f\|_{L^2}}{2\sqrt{\pi\,\mathrm{Im}  z}}.
\end{align*}
This completes the proof.
\end{proof}

The next proposition shows that the resolvent of the Lax operator preserves $\operatorname{Dom}(X^*)$.

\begin{proposition}[Resolvent invariance of $D(X^*)$]\label{lem:resolvent-dxstar}
Let $f\in \text{Dom}(X^*)$ and let $z\in \C\setminus \sigma(L_{u_0})$. Define $g:=(L_{u_0}-z)^{-1}f.$
Then $g\in \text{Dom}(X^*)$.
\end{proposition}
\begin{proof}
By the definition of $\text{Dom}(X^*)$,
it is enough to prove that $\widehat g|_{(0,\infty)}\in H^1(0,\infty)$.
Taking the Fourier transform of
$$
(L_{u_0}-z)g=(D-T_{u_0}-z)g=f,
$$
we obtain, for every $\xi>0$,
\begin{equation}\label{eq:g-h_u0}
(\xi-z)\widehat g(\xi)=\widehat f(\xi) + 
\widehat{u_0g}(\xi).
\end{equation}
Let us rewrite it as 
\begin{equation}\label{eq:g-h_A}
\widehat g(\xi)= \frac{\widehat f(\xi)}{\xi-z} + \frac{\widehat{u_0g}(\xi)}{\xi-z} =h_z(\xi)+\mathbb{A}_{u_0}^{(z)}\widehat g(\xi),	\qquad  \xi>0,
\end{equation}
where
$$
h_z(\xi):=\frac{\widehat f(\xi)}{\xi-z},\qquad
\mathbb{A}_v^{(z)}\psi(\xi):=
\frac{1}{2\pi(\xi-z)}
\int_0^\infty \widehat v(\xi-\zeta)\psi(\zeta)\,d\zeta,
\quad \xi>0.
$$
Since $f\in D(X^*)$, we know that $\widehat f|_{(0,\infty)}\in H^1(0,\infty)$.  Because $z\notin\mathbb R$, we thus have
$h_z\in H^1(0,\infty)$. In particular, $h_z$ is continuous on $[0,\infty)$. Moreover, since $u_0,g\in L^2(\mathbb R)$, we have $u_0g\in L^1(\mathbb R)$, so $\widehat{u_0g}$ is continuous on $\mathbb R$. It follows from \eqref{eq:g-h_A} that $\widehat g$ is continuous on $[0,\infty)$.

We also note that the operator $\mathbb{A}_v^{(z)}$ is bounded on $L^2(0,\infty)$. Indeed, let $\widetilde\psi$ denote the extension of $\psi$ by $0$ to the negative half-line. Then,
by Young's inequality,
$$
\|\mathbb{A}_v^{(z)}\psi\|_{L^2(0,\infty)}
\le
\frac{1}{2\pi}\Bigl\|\frac{1}{\cdot-z}\Bigr\|_{L^2(0,\infty)}
\|\widehat v*\widetilde\psi\|_{L^\infty(\mathbb R)}
\le
C_z\|v\|_{L^2(\mathbb R)}\|\psi\|_{L^2(0,\infty)},
$$
for some constant $C_z>0$ depending only on $z$.

To prove that $\widehat g|_{(0,\infty)}$ belongs to $H^1(0, \infty)$, it suffices to show that the difference quotient $\tau_\eta \widehat g$ converges in $L^2 (0, \infty)$ as $\eta \to 0$, where
$$
\tau_\eta\psi(\xi):=\frac{\psi(\xi+\eta)-\psi(\xi)}{\eta},
\qquad \eta>0.
$$
Now, let us introduce the truncated function 
$$
	u_0^R:=\mathbf 1_{|x|<R}u_0,
$$
and choose $R>0$ large enough such that
$ 
C_z\|u_0-u_0^R\|_{L^2}<\frac12$. Then $\mathrm{Id}-\mathbb{A}_{u_0-u_0^R}^{(z)}$ is invertible on $L^2(0,\infty)$, and equation \eqref{eq:g-h_A} can be rewritten as 
\begin{equation*}
\widehat g-\mathbb{A}_{u_0-u_0^R}^{(z)}\widehat g
=
h_z+\mathbb{A}_{u_0^R}^{(z)}\widehat g.
\end{equation*}
Applying $\tau_\eta$ to this identity, we obtain
\begin{equation}\label{eq:taugdeomp}
	\tau_\eta\widehat g
	-
	\mathbb{A}_{u_0-u_0^R}^{(z)}(\tau_\eta\widehat g)
	=
	\tau_\eta h_z
	+
	\tau_\eta\!\bigl(\mathbb{A}_{u_0^R}^{(z)}\widehat g\bigr)
	+
	[\tau_\eta,\mathbb{A}_{u_0-u_0^R}^{(z)}]\widehat g.
\end{equation}
We now analyze the three terms on the right-hand side. First, since $h_z\in H^1(0,\infty)$,  we have
\begin{equation}\label{eq:convhz}
\tau_\eta h_z  \xrightarrow{\eta\to0^+}
h_z' \qquad\text{in }L^2(0,\infty).
\end{equation}
Second, since $u_{0}^{R}$ has compact support, we have
$
x\,u_{0}^R\in L^2(\mathbb R),$
hence
$
(\widehat{u_{0}^R})'=-i\,\widehat{x\, u_{0}^R}\in L^2(\mathbb R).$ Therefore,
\begin{align*}
\partial_\xi\!\bigl(\mathbb{A}_{u_0^R}^{(z)}\widehat g\bigr)(\xi)
& =
-\frac{\mathbb{A}_{u_0^R}^{(z)}\widehat g(\xi)}{\xi-z}
+
\frac{1}{2\pi(\xi-z)}
\int_0^\infty (\widehat{u_0^R})'(\xi-\zeta)\widehat g(\zeta)\,d\zeta \\
& = -\frac{\mathbb{A}_{u_0^R}^{(z)}\widehat g(\xi)}{\xi-z}
- i \mathbb{A}_{x u_0^R}^{(z)}\widehat g \;
 \in L^2(0,\infty).
\end{align*}
Hence
$\mathbb{A}_{u_0^R}^{(z)}\widehat g\in H^1(0,\infty)$,
and therefore
\begin{equation}\label{eq:convtaueta}
\tau_\eta\!\bigl(\mathbb{A}_{u_0^R}^{(z)}\widehat g\bigr)
 \xrightarrow{\eta\to0^+}
-\frac{\mathbb{A}_{u_0^R}^{(z)}\widehat g(\xi)}{\xi-z}
- i \mathbb{A}_{x u_0^R}^{(z)}\widehat g 
\qquad\text{in }L^2(0,\infty).
\end{equation}
Third, we treat the commutator term.
Set
$
H(\xi):=\int_0^\infty \widehat v(\xi-\zeta)\psi(\zeta)\,d\zeta
$. Using the definition of $\tau_\eta$ and $\mathbb{A}_v^{(z)}$, we compute
\begin{align*}
[\tau_\eta,\mathbb{A}_v^{(z)}]\psi
& =
\tau_\eta\!\bigl(\mathbb{A}_v^{(z)}\psi\bigr)-\mathbb{A}_v^{(z)}(\tau_\eta\psi)\\
& = \frac{1}{2\pi\eta}
\left(
\frac{H(\xi+\eta)}{\xi+\eta-z}
-
\frac{H(\xi)}{\xi-z}
\right) - \frac{1}{2\pi(\xi-z)\eta}
\left[
\int_0^\infty \widehat v(\xi-\zeta)\psi(\zeta+\eta)\,d\zeta
-
H(\xi)
\right]\\
& = \frac{1}{2\pi\eta}
\frac{H(\xi+\eta)}{\xi+\eta-z} - \frac{1}{2\pi(\xi-z)\eta}
\int_0^\infty \widehat v(\xi-\zeta)\psi(\zeta+\eta)\,d\zeta.
\end{align*}
Making the change of variable $s=\zeta+\eta$, we get
\begin{align*}
	\int_0^\infty \widehat v(\xi-\zeta)\psi(\zeta+\eta)\,d\zeta
	&=
	\int_\eta^\infty \widehat v(\xi+\eta-s)\psi(s)\,ds \\
	&=
	I(\xi+\eta)-\int_0^\eta \widehat v(\xi+\eta-s)\psi(s)\,ds.
\end{align*}
Therefore, for every $v\in L^2(\mathbb R)$ and every continuous
$\psi$ on $[0,\infty)$, we have
\begin{align}\label{eq:taucomdecom}
[\tau_\eta,\mathbb{A}_v^{(z)}]\psi(\xi)
&=
\frac{H(\xi+\eta)}{2\pi\eta}
	\left(
	\frac{1}{\xi+\eta-z}-\frac{1}{\xi-z}
	\right)
	+
\frac{1}{2\pi(\xi-z)}
	\int_0^\eta \widehat v(\xi+\eta-s)\psi(s)\,\frac{ds}{\eta} \nonumber\\
& = -\frac{\mathbb{A}_v^{(z)}\psi(\xi+\eta)}{\xi-z}
+
\frac{1}{2\pi(\xi-z)}
\int_0^\eta
\widehat v(\xi+\eta-\zeta)\psi(\zeta)\,\frac{d\zeta}{\eta}.
\end{align}
where we used 
$
\frac{1}{\xi+\eta-z}-\frac{1}{\xi-z}
=
-\frac{\eta}{(\xi+\eta-z)(\xi-z)}.$

We come back to  \eqref{eq:taugdeomp}. 
Since $\widehat g$ is continuous on $[0,\infty)$, we have
$$
-\frac{\mathbb A_v^{(z)}\widehat g(\xi+\eta)}{\xi-z} 
 \xrightarrow{\eta\to0^+} -\frac{\mathbb A_v^{(z)}\widehat g}{\xi-z}, \qquad \text{in}\; L^2(0,\infty).
$$
For the second term on the right-hand side of \eqref{eq:taucomdecom}, we compute
\begin{align*}
 & \left\|\frac{1}{2\pi(\cdot-z)}
\int_0^\eta
\widehat v(\cdot+\eta-\zeta)\widehat g(\zeta)\,\frac{d\zeta}{\eta}-\frac{\widehat v(\cdot) \,\widehat g(0^+)}{2\pi(\cdot-z)}\right\|_{L^2(0,\infty)}\\ & \le
 C_z\frac1\eta\int_0^\eta
\|\widehat g(\zeta)\,
\widehat v(\cdot+\eta-\zeta)-\widehat g(0^+)\widehat v(\cdot)\|_{L^2(\mathbb R)}
\,d\zeta \\
& \leq C_z \sup_{0\le \zeta\le \eta} \|\widehat g(\zeta)\widehat v(\cdot+\eta-\zeta)-\widehat g(0^+)\widehat v(\cdot)\|_{L^2(\mathbb R)}
\to 0, \qquad \eta\to0^+.
\end{align*}
Consequently,
$$
[\tau_\eta, \mathbb A_v^{(z)}]\widehat g
 \xrightarrow{\eta\to0^+}
-\frac{\mathbb A_v^{(z)}\widehat g(\xi)}{\xi-z}
+
\frac{\widehat v(\xi)\widehat g(0^+)}{2\pi(\xi-z)},
\qquad\text{in }L^2(0,\infty).
$$
Applying this with $v=u_0-u_0^R$, we get
\begin{equation}\label{eq:convtaucommu}
[\tau_\eta,\mathbb{A}_{u_0-u_0^R}^{(z)}]\widehat g
 \xrightarrow{\eta\to0^+}
-\frac{\mathbb A_{u_0-u_0^R}^{(z)}\widehat g(\xi)}{\xi-z}
+
\frac{ (\widehat u_0- \widehat u_0^R)(\xi)\widehat g(0^+)}{2\pi(\xi-z)},
\qquad\text{in }L^2(0,\infty).
\end{equation}
Combining \eqref{eq:taugdeomp}--\eqref{eq:convtaucommu} with the invertibility of $\mathrm{Id}-\mathbb{A}_{u_0-u_0^R}^{(z)}$ on $L^2(0,\infty)$, it follows that $\tau_\eta\widehat g$ possesses a limit in $L^2(0,\infty)$ as $\eta\to0^+$, defined as:
$$ (\mathrm{Id}-\mathbb{A}_{u_0-u_0^R}^{(z)})^{-1} \Big(h_z' -\frac{\mathbb{A}_{u_0^R}^{(z)}\widehat g(\xi)}{\xi-z}
- i \mathbb{A}_{x u_0^R}^{(z)}\widehat g (\xi) -\frac{\mathbb A_{u_0-u_0^R}^{(z)}\widehat g(\xi)}{\xi-z}
+
\frac{ (\widehat u_0- \widehat u_0^R)(\xi)\widehat g(0^+)}{2\pi(\xi-z)}\Big).
$$
Therefore,
$
\widehat g|_{(0,\infty)}\in H^1(0,\infty)$ and we conclude that $g\in \text{Dom}(X^*)$.
\end{proof}

The next lemma is a density result for the shifted operator
$(n+1)L_{u_0}^n+c$.	
\begin{lemma}\label{lem:density-Ac-DXstar}
Let $
	A_c:=(n+1)L_{u_0}^n+c. $
Then
$$
\overline{A_c\bigl(\operatorname{Dom}(X^*)\cap \operatorname{Dom}(A_c)\bigr)}
=
\ker(A_c)^\perp.
$$
\end{lemma}

\begin{proof}
It suffices to show
$$
[A_c\bigl(D(X^*)\cap D(A_c)\bigr)]^\perp=\ker(A_c).
$$
The inclusion
$
\ker(A_c)\subset
\bigl[A_c\bigl(\operatorname{Dom}(X^*)\cap \operatorname{Dom}(A_c)\bigr)\bigr]^\perp
$
is immediate from the self-adjointness of $A_c$.
Conversely, let $h\in [A_c\bigl(D(X^*)\cap D(A_c)\bigr)]^\perp$. We show that $h\in \ker(A_c)$.
Applying Proposition~\ref{lem:resolvent-dxstar} with
$
z=-\frac{1}{\varepsilon}$, where \(\varepsilon>0\) is chosen so that
$$
1+\varepsilon\lambda_0\ge \frac12,
\qquad \lambda_0:=\inf\sigma(L_{u_0}),
$$
it then follows that
$$
(\mathrm{Id}+\varepsilon L_{u_0})^{-1}
:\operatorname{Dom}(X^*)\to \operatorname{Dom}(X^*),
$$ 
hence, by iteration,
$$
(\mathrm{Id}+\varepsilon L_{u_0})^{-n}\bigl(\operatorname{Dom}(X^*)\bigr)
\subset
\operatorname{Dom}(X^*).
$$
Moreover, by the spectral theorem, the operator
$
A_c(\mathrm{Id}+\varepsilon L_{u_0})^{-n}
$
is bounded and self-adjoint on $L^2_+(\mathbb R)$, since its spectral multiplier
$
\lambda\longmapsto ((n+1)\lambda^n+c ) (1+\varepsilon\lambda)^{-n}
$
is bounded and real-valued. In particular,
$$
(\mathrm{Id}+\varepsilon L_{u_0})^{-n}f\in \operatorname{Dom}(A_c), \quad \text{for every } f\in L^2_+(\mathbb R)
$$
Therefore, for every
$
f\in \operatorname{Dom}(X^*),
$
we have
$$
(\mathrm{Id}+\varepsilon L_{u_0})^{-n}f\in
\operatorname{Dom}(X^*)\cap \operatorname{Dom}(A_c).
$$
Since $h$ is orthogonal to
$
A_c\bigl(\operatorname{Dom}(X^*)\cap \operatorname{Dom}(A_c)\bigr),$
it follows that
$$
0=\langle h,\,
A_c(\mathrm{Id}+\varepsilon L_{u_0})^{-n}f\rangle
=
\langle A_c(\mathrm{Id}+\varepsilon L_{u_0})^{-n}h,\,
f\rangle
$$
for every $f\in \operatorname{Dom}(X^*)$.
Because $\operatorname{Dom}(X^*)$ is dense in $L^2_+(\mathbb R)$, we conclude that
$$
A_c(\mathrm{Id}+\varepsilon L_{u_0})^{-n}h=0.
$$
Finally, by combining the convergence 
$
(\mathrm{Id}+\varepsilon L_{u_0})^{-n}h \to  h$ in $L^2_+(\mathbb R)$ as $\varepsilon\to0^+,$
with the closedness of $A_c$, we obtain
$$
h\in \operatorname{Dom}(A_c)
\qquad\text{and}\qquad
A_ch=0.
$$
This implies
$
h\in \ker(A_c).$
Therefore, we have
$
\bigl[A_c\bigl(\operatorname{Dom}(X^*)\cap \operatorname{Dom}(A_c)\bigr)\bigr]^\perp=\ker(A_c).
$
The desired statement follows immediately by taking orthogonal complements.
\end{proof}

The next proposition is the main vanishing statement for the asymptotic analysis. 
It shows that, after shifting by the speed $c$, the quantity $\Omega_t^{(n)}f(z+ct)$ tends to zero whenever $f$ is orthogonal to the kernel of the shifted operator.
\begin{proposition}\label{prop:general-kernel-shift}
	Let $c\in \R$ and $f\in L^2_+(\R)$ be such that
	$
	f\perp \ker\bigl((n+1)L_{u_0}^n+c\bigr).$
	Then, for every $z\in \C_+$,
	$$
	\Omega_t^{(n)}f(z+ct) \longrightarrow 0,
	\qquad \text{as } t\to \infty.$$
\end{proposition}

\begin{proof}
Set $A_c := (n+1)L_{u_0}^n + c.$ We first treat the case where
$$f=A_c g \qquad \text{with} \quad g\in  \operatorname{Dom}(X^*) \cap  \operatorname{Dom}(A_c).$$ Using the identity
$$
A_cg
=
-\frac1t\bigl(X^*-tA_c-z\bigr)g
+\frac1t(X^*-z)g,$$
we obtain
\begin{align*}
\Omega_t^{(n)}f(z+ct)
=
		\frac{1}{2\ii\pi}
		I_+\!\left(
		\bigl(X^*-tA_c-z\bigr)^{-1}A_cg
		\right) 
		= -\frac{1}{2\ii\pi t}I_+(g)
		+ \frac1t\,\Omega_t^{(n)}\bigl((X^*-z)g\bigr)(z+ct).
\end{align*}
Since $g\in  \operatorname{Dom}(X^*)$, the quantity $I_+(g)=\widehat g(0^+)$ is well-defined and $(X^*-z)g\in L^2_+(\R)$. Hence, by Lemma~\ref{L2boundexplicit},
$$
\left|\Omega_t^{(n)}\bigl((X^*-z)g\bigr)(z+ct)\right|
\le
\frac{\|(X^*-z)g\|_{L^2}}{2\sqrt{\pi\,\operatorname{Im} z}}.
$$
It follows that
$$
\left|\Omega_t^{(n)}f(z+ct)\right|
\le
\frac{|I_+(g)|}{2\pi\,t}
+
\frac{\|(X^*-z)g\|_{L^2}}{2t\sqrt{\pi\,\operatorname{Im} z}},
$$
and therefore
$$
	\Omega_t^{(n)}f(z+ct)\longrightarrow 0
	\qquad\text{as }t\to\infty.$$
Now, for the case that
$
f\perp \ker(A_c).$
the proof is completed by using Lemma \ref{L2boundexplicit} and the density result in Lemma~\ref{lem:density-Ac-DXstar}.
\end{proof}	

The proposition below gives the soliton limit, which is the key ingredient in the proof of the traveling-wave classification theorem. 
\begin{proposition}[Soliton limit]\label{prop:Solitonlimit}
Let $\lambda$ be a simple negative eigenvalue of the Lax operator $L_{u_0}$ on the Hardy space $L^2_+(\R)$, with normalized eigenfunction $\phi$. Then, for every $z\in\C_+$,
\begin{align}\label{eq:convergPi}
\Pi u\bigl(t,z-(n+1)\lambda^n t\bigr)
\longrightarrow
\frac{\ii}{z-\langle X^*\phi,\phi\rangle},
\qquad\text{as }t\to \infty.
\end{align}
\end{proposition}

\begin{proof}
Recall the formula $
\Pi u(t,\cdot)=\Omega_t^{(n)}\Pi u_0$, it suffices to study
$
\Omega_t^{(n)}\Pi u_0\bigl(z-(n+1)\lambda^n t\bigr).
$
We decompose
$ \Pi u_0=\langle \Pi u_0,\phi\rangle \phi+(\Pi u_0)_{ \perp}$, where 
$(\Pi u_0)_{\perp} \perp \phi$. Apply Proposition~\ref{prop:general-kernel-shift} with $ c:=-(n+1)\lambda^n$,
we have, for every $z\in \C_+$, 
\begin{equation}\label{eqconverperd}
\Omega_t^{(n)}(\Pi u_0)_{ \perp}(z -(n+1)\lambda^n t)\longrightarrow 0, \quad \text{as }t\to\infty.
\end{equation}
Next, define
$
\psi:=
\frac{\phi}{\langle X^*\phi,\phi\rangle-z}. 
$
Since $L_{u_0}^n\phi=\lambda^n\phi,$ we have
$$
(X^*-t(n+1) (L_{u_0}^n - \lambda^n )-z)\psi =\phi+\rho,
\quad \text{with} \quad
\rho:=
\frac{X^*\phi-\langle X^*\phi,\phi\rangle\phi}
{\langle X^*\phi,\phi\rangle-z}.
$$
Equivalently,
$$
[X^*-t (n+1) (L_{u_0}^n - \lambda^n )-z]^{-1}\phi=\psi-[X^*- t (n+1) (L_{u_0}^n - \lambda^n ) -z]^{-1} \rho.
$$
Applying $I_+$ to both sides of the identity above and using the definition of $\Omega_t^{(n)}$, 
we obtain
	$$
	\Omega_t^{(n)}\phi(z -(n+1)\lambda^n t)
	=
	\frac{I_+(\psi)}{2\ii\pi}
	-\Omega_t^{(n)}\rho (z -(n+1)\lambda^n t).
	$$
It is easy to see that $\rho\perp \phi,$ hence
Proposition ~\ref{prop:general-kernel-shift} implies
$
\Omega_t^{(n)}\rho(z -(n+1)\lambda^n t)  \xrightarrow{t\to\infty } 0.
$
Therefore,
	$$
	\Omega_t^{(n)}\phi(z -(n+1)\lambda^n t)
	\longrightarrow
	\frac{I_+(\phi)}
	{2\ii\pi\bigl(\langle X^*\phi,\phi\rangle-z\bigr)}, \qquad\text{as }t\to\infty. 
	$$
Combining this with the decomposition of $\Pi u_0$ and the convergence \eqref{eqconverperd}, we obtain 
\begin{align}\label{eq:lambdafn}
\Omega_t^{(n)}\Pi u_0(z-(n+1)\lambda^n  t)
 \longrightarrow
\frac{\langle \Pi u_0,\phi\rangle\,I_+(\phi)}
{2\ii\pi\bigl(\langle X^*\phi,\phi\rangle-z\bigr)}, \qquad\text{as }t\to\infty. 
\end{align}
Finally, by Wu's identities \cite[Lemma 2.5]{wu2016simplicity} generalized recently in 
\cite{badreddine2025orbital, gassot2026infinite},
\begin{equation}\label{eq:Wu} |\langle \phi, \Pi u_0\rangle |^2=-2\pi \lambda\ ,\ \lambda I_+(\phi)=-\langle \phi,\Pi u_0\rangle ,
\end{equation}
we have
$$\frac{\langle \Pi u_0, \phi \rangle I_+(\phi)}{2 \ii \pi} = - \frac{\langle \Pi u_0, \phi \rangle \langle \phi, \Pi u_0\rangle }{2 \ii \pi \lambda } = -\ii. $$
Substituting this identity into \eqref{eq:lambdafn}, we get the desired result. 
\end{proof}

Now, we are ready to prove Theorem \ref{thm:classification}. 

\begin{proof}[Proof of Theorem~\ref{thm:classification}]
Since $u(t,x)=u_0(x-ct),$ by the definition of $\Omega_t^{(n)}$, we have
$$
\Pi u_0(z)=\Pi u(t,z+ct)=\Omega_t^{(n)}(\Pi u_0)(z+ct), \quad \text{for every} \quad z\in\C_+, t\in\R.
$$
We first consider the case $c=0$. Let
$\varphi\in \ker\bigl((n+1)L_{u_0}^n\bigr)$. Since $L_{u_0}$ is self--adjoint,   $\varphi\in\ker(L_{u_0})$, hence 
$D \varphi=T_{u_0}\varphi=\Pi(u_0 \varphi).$
Taking Fourier transforms, we obtain, for every $\xi>0$,
	$$
	\xi \widehat{\varphi}(\xi)=\widehat{u_0\varphi}(\xi).
	$$
Since $u_0,\phi\in L^2(\R)$, we have $u_0\phi\in L^1(\R)$, so $\widehat{u_0\phi}$ is continuous. 
Passing to the limit as $\xi\to0^+$ yields
$\int_{\R} u_0(y)\varphi(y)\,dy=0.$
Now, write $u_0=\Pi u_0+\overline{\Pi u_0}$.  
Since $\Pi u_0,\varphi\in L^2_+(\R)$, their Fourier transforms are supported in $[0,\infty)$, and therefore
$$
\int_{\R} (\Pi u_0)(y)\varphi(y)\,dy
=0.
$$
It follows that
$0=\langle \varphi,\Pi u_0\rangle,$ hence
$\Pi u_0\perp \ker\bigl((n+1)L_{u_0}^n\bigr).$
By Proposition~\ref{prop:general-kernel-shift}, for every $z\in\C_+$, 
$$\Pi u_0(z)=\Omega_t^{(n)}(\Pi u_0)(z)\longrightarrow 0
\qquad\text{as }t\to\infty,
$$
which yields $\Pi u_0\equiv0$. Since $u_0$ is real-valued, we have $u_0\equiv0$.
	
Now, we assume that $u_0\not\equiv0$. By the first part of the proof, we necessarily have $c\neq0$. Moreover, 
$\ker\bigl((n+1)L_{u_0}^n+c\bigr)\neq\{0\}.$ Otherwise, we have $\Pi u_0\perp \ker\bigl((n+1)L_{u_0}^n+c\bigr)$, which contradicts with $u_0\not\equiv0$ in view of Proposition \ref{prop:general-kernel-shift}.
Since eigenvalues of $L_{u_0}^n$ are $n$--th powers of eigenvalues of $L_{u_0}$, there exists a negative eigenvalue $\lambda$ of $L_{u_0}$ such that
$c=-(n+1)\lambda^n.$
Applying Proposition~\ref{prop:Solitonlimit}, we obtain
$$
\Pi u_0(z)
=
\lim_{t\to\infty}\Pi u\bigl(t,z-(n+1)\lambda^n t\bigr)
=
\frac{\ii}{z-\langle X^*\phi,\phi\rangle}
= \frac{\ii}{z+p},$$
where $p:=-\langle X^*\phi,\phi\rangle.$
Therefore,
$$
u_0(x)
=
\frac{\ii}{x+p}-\frac{\ii}{x+\overline{p}}
=
\frac{2\mathrm{Im}  p}{|x+p|^2}
=
R_{p}(x).
$$
Combining the Fourier representation 
$\mathrm{Im} \langle X^*\phi,\phi\rangle
= 
-\frac{|I_+(\phi)|^2}{4\pi}
$
with identities \eqref{eq:Wu},  we have
\begin{equation*}
	\mathrm{Im}\, p=\frac{\vert I_+(\phi)\vert ^2}{4\pi}=\frac{1}{2|\lambda|}>0.
\end{equation*}
Since $\lambda<0$, we obtain $\lambda=-\frac{1}{2 \mathrm{Im} p},$
and therefore
$
c=-(n+1)\lambda^n
=
(-1)^{n+1}\frac{n+1}{(2 \mathrm{Im} p)^n}.
$
This completes the proof.
\end{proof}

We conclude this section by checking that the profiles appearing in Theorem~\ref{thm:classification} are indeed  traveling-wave solutions of the $(n+1)$-th BO flow.
We observe that $\Pi R_p(z)=\frac{i}{z+p},  z\in\C_+,$ and that $\Pi R_p$ is an eigenfunction of $L_{R_p}$ associated with the eigenvalue $\lambda=-\frac1{2\mathrm{Im} p}.$
Applying the explicit formula \eqref{BOhierexpfor} of Theorem~\ref{mainBOhier} with $u_0=R_p$, we obtain
\begin{align*}
	\Pi u(t,z)
	&=
	\frac{1}{2\pi i}
	I_+\bigl(\bigl(X^*-(n+1)tL_{R_p}^n-z\bigr)^{-1}\Pi R_p\bigr) \\
	&=
	\frac{1}{2\pi i}
	I_+\bigl(\bigl(X^*-(n+1)t\lambda^n-z\bigr)^{-1}\Pi R_p\bigr) \\
&	=
	\frac{i}{z+p+(n+1)t\lambda^n} 
	=
	\frac{i}{z-c_{n,p}t+p} 
	=
	\Pi\bigl(R_p(\,\cdot-c_{n,p}t)\bigr)(z).
\end{align*}
where
$
c_{n,p}
=(-1)^{n+1}\frac{n+1}{(2\mathrm{Im} p)^n}.$
Therefore,
$$
u(t,x)=R_p(x-c_{n,p}t)
$$
is a traveling-wave solution of the $(n+1)$-th Benjamin--Ono flow. In particular, Theorem~\ref{thm:classification} shows that, on the real line, the nontrivial traveling waves are exactly the one-solitons.

\section{Zero dispersion limit for third order BO}\label{seczeroBo3}
In this section, we consider the third equation of the BO hierarchy with a small dispersion parameter $\varepsilon > 0$,
\begin{equation}\label{BO3eps}
	\begin{cases}
		\partial_t u^\varepsilon + \varepsilon^2 \partial_x^3 u^\varepsilon + \frac{3}{2} \varepsilon \partial_x (u^\varepsilon |D| u^\varepsilon) + \frac{3}{2} \varepsilon |D| (u^\varepsilon \partial_x u^\varepsilon) - \partial_x (u^\varepsilon)^3 = 0,\\
		u^\varepsilon (0, x) = u_{0} (x) \in L^2_{r}(\mathbb{R}),
	\end{cases}	
\end{equation}
In Hamiltonian form, it reads
\begin{equation*}
\partial_t u^\varepsilon = \partial_x (\nabla E^\varepsilon_2(u^\varepsilon)), \quad \text{with} \;\; E^\varepsilon_2(u^\varepsilon) := \langle (\varepsilon D - T_{u^\varepsilon})^{2} \Pi u^\varepsilon, \Pi u^\varepsilon \rangle_{L^2}
\end{equation*}
We introduce the re-scaled variables:
$$
u^\varepsilon(t, x) = \varepsilon\, v(s, x), \quad t = \frac{s}{\varepsilon^2}.
$$
Under this scaling, the function $v(s, x)$ satisfies the third order BO equation:
\begin{equation*}
	\begin{cases}
		\partial_s v
		+ \partial_x^3 v
		+ \frac{3}{2} \partial_x(v |D| v)
		+ \frac{3}{2} |D|(v \partial_x v)
		- \partial_x(v)^3 = 0, \\
		v(0, x) = v_0(x) = \varepsilon^{-1} u_{0}(x).
	\end{cases}
\end{equation*}
From Section \ref{secproof}, we know that, if $v_0 \in H^{s}_r (\mathbb{R})$, $s > 5/2$, the solution of the above initial value problem is given by 
$$
v(s,x) = \Pi v(s,x) + \overline{\Pi v(s,x)}, 
\qquad x \in \mathbb{R},
$$
where for every complex number $z$ with $\mathrm{Im}(z)>0$,
\begin{align}\label{formula_BO3}
	\Pi v(s, z) 
	&= \frac{1}{2 \ii \pi} \, I_+ \left[ \left( X^* - 3s L_{v_0}^2 - z \mathrm{Id} \right)^{-1} \Pi v_0 \right] \\
	& = \frac{1}{2 \ii \pi} \, I_+ \left[ \left( X^* - 3s D^2 + 3s (D T_{v_0} + T_{v_0}D )- 3s T_{v_0}^2  - z \mathrm{Id} \right)^{-1} \Pi v_0 \right]\nonumber.
\end{align}
We know from Remark \ref{remark:well-def} that, the resolvent on the right hand side is well defined, and 
$$e^{- \ii \frac{t}{\varepsilon} (\varepsilon D - T_{u_0} )^3}  U^*(t) X^* U(t)  e^{ \ii \frac{t}{\varepsilon} (\varepsilon D - T_{u_0} )^3}
= X^* -3 t (\varepsilon D - T_{u_0})^2.$$
As the maximal dissipativity
is preserved under unitary conjugation, and 
$U^*(t) X^* U(t)$ is maximally dissipative, we find that the operator on the right-hand side is also maximally dissipative. 

Moreover, scaling the solution back via $\varepsilon \Pi v(s,x) = \Pi u^\varepsilon(t,x)$ and $\varepsilon \Pi v_0 = \Pi u_0$, we can write the solution to the original initial value problem \eqref{BO3eps} as
$$
u^\varepsilon(t,x) = \Pi u^\varepsilon(t,x) + \overline{\Pi u^\varepsilon(t,x)}, 
\qquad x \in \mathbb{R},$$
where  for every complex number $z$ with $\mathrm{Im}(z)>0$,
\begin{align*}	
	\Pi u^\varepsilon(t, z)  
	& = \frac{1}{2 \ii \pi} \, I_+ \left[ \left( X^* - 3t (\varepsilon D-T_{u_0})^2 - z \mathrm{Id} \right)^{-1} \Pi u_0 \right] \\
	& = \frac{1}{2 \ii \pi} \, I_+ \left[ \left( X^* - 3t \varepsilon^2 D^2 + 3t \varepsilon (D T_{u_0} + T_{u_0}D )- 3t T_{u_0}^2  - z \mathrm{Id} \right)^{-1} \Pi u_0 \right] \\
	& = \frac{1}{2 \ii \pi} \, I_+ \left[ \left( X^* - 3t \varepsilon^2 D^2 + 6t \varepsilon D T_{u_0} - 3t \varepsilon T_{D u_0} - 3t T_{u_0}^2  - z \mathrm{Id} \right)^{-1} \Pi u_0 \right].
\end{align*}
The maximal dissipativity of $ X^* - 3t (\varepsilon D-T_{u_0})^2$  guarantees that the resolvent on the right hand side is well defined. 

Using the explicit formula from Theorem \ref{mainBO3}, we can prove that, for the third equation of the BO hierarchy with a small dispersion parameter $\varepsilon > 0$, 
the solution $u^\varepsilon(t)$ of \eqref{BO3eps} converges weakly to a limit in $L^2(\mathbb{R})$, and the limit is given by a geometric formula. 
\begin{theorem}\label{thm: mainBO3eps}
Let $u_0 \in H^s_r(\mathbb{R})$ with $s >  \frac{5}{2}$.  For every $t \in \mathbb{R}$, the solution $u^\varepsilon(t)$ of \eqref{BO3eps} converges weakly in $L^2(\mathbb{R})$ to a function $ZD[u_0](t, \cdot)$.
The set $K_t(u_0)$ of critical values of the function
	$$
	y \in \mathbb{R} \mapsto y - 3 t u_0^2(y)
	$$
	is a compact subset of measure $0$. For every connected component $\Omega$ of $K_t(u_0)$, there exists a nonnegative integer $l$ such that, for every $x\in \Omega$, the equation
	$$
	y - 3 t u_0^2(y) = x
	$$
	has $2 \ell +1$ simple real solutions 
	$$
	y_0(t, x ) < y_1(t, x ) < \cdots < y_{2 \ell}(t, x )
	$$
	and the zero dispersion limit is given by 
	\begin{align}\label{eq:ZDu_0geom}
		\textrm{ZD}[u_0](t, x) = \sum_{k = 0}^{2 \ell} (-1)^k u_0 (y_k (t, x)).
	\end{align}
\end{theorem}

Before giving the proof of Theorem \ref{thm: mainBO3eps}, we first state the following weak convergence result and the stability statement that will be used in the density argument below.
\begin{proposition}\label{prop:BO3-weak}
Let $u_0 \in H^s_r(\mathbb{R})$ with $s >  \frac{5}{2}$. For every $t \in \mathbb{R}$, the solution $u^\varepsilon(t)$ of \eqref{BO3eps} converges weakly in $L^2(\mathbb{R})$ to a function $ZD[u_0](t, \cdot)$, characterized by
	$$
	\forall x \in \mathbb{R}, \quad ZD[u_0](t, x) = \Pi ZD[u_0](t, x) + \overline{\Pi ZD[u_0](t, x)},
	$$
	where, for every complex number $z$ with $\textrm{Im} \,z >0$, 
	\begin{align}\label{eq:BO3T2}
		\Pi ZD[u_0](t, z) = \frac{1}{2 \ii \pi }\, I_+ \Bigl[(X^* - 3t\,T_{u_0}^2 - z\,\mathrm{Id})^{-1} (\Pi u_0)\Bigr].
	\end{align}
Moreover, the map
$$
u_0 \in L^\infty(\mathbb{R}) \cap L^2(\mathbb{R}) \longmapsto ZD[u_0]
$$
defined by Proposition \ref{prop:BO3-weak} has the following properties.
\begin{enumerate}
	\item For every $t \in \mathbb{R}$,
	$
	\|ZD[u_0](t)\|_{L^2} \le \|u_0\|_{L^2},
	$
	and $ZD[u_0]$ is continuous on $\mathbb{R}$ with values in $L^2(\mathbb{R})$ endowed with the weak topology.
	\item If $u_0^\delta$ converges strongly to $u_0$ in $L^2(\mathbb{R})$ as $\delta \to 0$, with a uniform bound on $\|u_0^\delta\|_{L^\infty}$, then, for every $t \in \mathbb{R}$, $ZD[u_0^\delta](t)$ converges to $ZD[u_0](t)$ weakly in $L^2(\mathbb{R})$.
\end{enumerate}
\end{proposition}

\begin{remark}
Since $u_0 \in L^\infty(\mathbb{R})$, the operator $T_{u_0}^2$ is bounded and self-adjoint. The maximal dissipativity of $-iX^*$ implies that the operator $-i(X^* - 3t\,T_{u_0}^2)$ is maximally dissipative, thus the formula  \eqref{eq:BO3T2} is well-defined. 
The same observation applies to the higher-order resolvent
formula in Section \ref{seczeroBon} (see the formula \eqref{eq:BOnTn} in Proposition \ref{prop:BOhier-weak}).
\end{remark}

To pass from Proposition \ref{prop:BO3-weak} to the geometric formula in Theorem \ref{thm: mainBO3eps},  we must analyze the operators inside the resolvent in \eqref{eq:BO3T2}. The following lemma collects the necessary compactness and convergence properties for the resolvent equations.
\begin{lemma}\label{lemma:perturbedbo3}
Let $z \in \mathbb{C}$ with $\mathrm{Im}(z) > 0$, and assume $u_0 \in H^s_r (\R) $, $s > 5/2$. Let $h_\varepsilon, h \in L^2(\mathbb{R_+})$ be the solutions to the two equations below, respectively, 
\begin{align}
	& \left( X^* - 3t \varepsilon^2 D^2 + 6t \varepsilon D T_{u_0} - 3t \varepsilon T_{D u_0} - 3t T_{u_0}^2  - z \mathrm{Id} \right) h_\varepsilon =\Pi u_0, \label{eq:pertubedequation3}\\
	& \left(X^* - 3t\,T_{u_0}^2 - z\,\mathrm{Id}\right)h =  \Pi u_0\label{eq:nonpertubedequation3}.
\end{align}
Then, as $\varepsilon \to 0$, we have
\begin{enumerate}
	\item  $h_\varepsilon$ is uniformly bounded in $L^2(\mathbb{R}_+)$;
	\item $\hat{h}_\varepsilon \rightharpoonup \hat{h}$ weakly in $L^2(\mathbb{R}_+)$;
	\item  $\hat{h}_\varepsilon(0) \to \hat{h}(0)$.
\end{enumerate}
\end{lemma}

\begin{proof}
Let us set $$
A_\varepsilon := X^* - 3t (\varepsilon D-T_{u_0})^2 = 
X^* - 3t \varepsilon^2 D^2 + 6 t \varepsilon D T_{u_0} - 3t \varepsilon T_{D u_0} - 3t T_{u_0}^2. $$
Because $A_\varepsilon$ is maximally dissipative, we have $\|(A_\varepsilon - z \mathrm{Id})^{-1}\| \leq C_z$. This immediately yields the uniform bound for the solution $h_\varepsilon$,
$$ \|\hat{h}_\varepsilon\|_{L^2} = \|h_\varepsilon\|_{L^2} \le \|(A_\varepsilon - z \Id )^{-1}\| \|\Pi u_0\|_{L^2} \le C_z \|\Pi u_0\|_{L^2} \leq C.$$
Thus, there exists a subsequence converging weakly to some $\hat{g} \in L^2(\mathbb{R}_+)$.
To identify $\hat{g}$, we take the Fourier space formulation of the  equation \eqref{eq:pertubedequation3},	
\begin{align}\label{eq:perturbed}
\ii \partial_\xi \hat{h}_\varepsilon - z \hat{h}_\varepsilon - 3t \widehat{T_{u_0}^2 h_\varepsilon} - \widehat{\Pi u_0} = \varepsilon \Big( 3t \varepsilon \xi^2 \hat{h}_\varepsilon - 6t \xi \widehat{T_{u_0} h_\varepsilon} + 3t \widehat{T_{D u_0} h_\varepsilon} \Big).
\end{align} 
Let $\phi \in C_c^\infty((0, \infty))$ be a test function. Multiplying the equation by $\phi$ and integrating over $\xi \in (0, \infty)$, we obtain
\begin{align}\label{eq:hvarepphi}
	\int_0^\infty \hat{h}_\varepsilon \Big( - \ii \partial_\xi \phi - z \phi \Big) d\xi - 3t \int_0^\infty \widehat{T_{u_0}^2 h_\varepsilon} \phi \, d\xi - \int_0^\infty \widehat{\Pi u_0} \phi \, d\xi = \mathcal{E}_\varepsilon(\phi),
\end{align} 
where the error term is given by
$$ \mathcal{E}_\varepsilon(\phi) = \varepsilon \int_0^\infty \hat{h}_\varepsilon \Big( 3t \varepsilon \xi^2 \phi \Big) d\xi - \varepsilon \int_0^\infty \widehat{T_{u_0} h_\varepsilon} \Big( 6t \xi \phi \Big) d\xi + \varepsilon \int_0^\infty \widehat{T_{D u_0} h_\varepsilon} \Big( 3t \phi \Big) d\xi. $$
Because $\phi$ has compact support, the multipliers $\xi^2 \phi$ and $\xi \phi$ are bounded in $L^\infty$. Combined with the uniform bound $\|\hat{h}_\varepsilon\|_{L^2} \le C$ and the $L^2$-boundedness of the Toeplitz operators $T_{u_0}$ and $T_{Du_0}$, we have by the Cauchy-Schwarz inequality that
$$ |\mathcal{E}_\varepsilon(\phi)| \le \varepsilon \cdot \tilde{C} \xrightarrow{\varepsilon \to 0} 0. $$
Passing to the limit $\varepsilon \to 0$ on both sides of \eqref{eq:hvarepphi}, we get
$$ \int_0^\infty \hat{g} \Big( - \ii \partial_\xi \phi - z \phi \Big) d\xi - 3t \int_0^\infty \widehat{T_{u_0}^2 g} \phi \, d\xi = \int_0^\infty \widehat{\Pi u_0} \phi \, d\xi. $$
Therefore, $\hat{g}$ is a weak solution to the equation 
$$(X^* - 3tT_{u_0}^2 - z\Id )g = \Pi u_0.$$
We conclude that $\hat{g} = \hat{h}$, 
thus, the entire sequence converges weakly, $\hat{h}_\varepsilon \rightharpoonup \hat{h}$. The second point is proved.
	
To prove the third point, let $\chi \in C_c^\infty([0, \infty))$ be a cut-off function with $\chi(0) = 1$. By integration by parts, $ \hat{h}_\varepsilon$ at $\xi = 0$ is given by
$$ \hat{h}_\varepsilon (0) = - \int_{0}^{\infty} \partial_\xi (\hat{h}_\varepsilon \chi) d\xi = - \int_{0}^{\infty} (\partial_\xi \hat{h}_\varepsilon) \chi d\xi - \int_{0}^{\infty} \hat{h}_\varepsilon \chi' d\xi. $$
Plugging the explicit expression for $\partial_\xi \hat{h}_\varepsilon$ from \eqref{eq:perturbed} into this formula gives
$$ \hat{h}_\varepsilon(0) = i \int_0^\infty \Big( \widehat{\Pi u_0} + z \hat{h}_\varepsilon + 3t \widehat{T_{u_0}^2 h_\varepsilon} \Big) \chi \, d\xi - \int_{0}^{\infty} \hat{h}_\varepsilon \chi' d\xi + I_\varepsilon, $$
where 
$$ I_\varepsilon = \ii \varepsilon \int_0^\infty \Big( 3t \varepsilon \xi^2 \hat{h}_\varepsilon - 6t \xi \widehat{T_{u_0} h_\varepsilon} + 3t \widehat{T_{D u_0} h_\varepsilon} \Big) \chi \, d\xi. $$
Using the same argument as before, the integral within $I_\varepsilon$ is bounded uniformly in $\varepsilon$, implying $|I_\varepsilon| \le C \varepsilon \to 0$.
Then, taking the limit $\varepsilon \to 0$ in the remaining terms, and using the weak convergence $\hat{h}_\varepsilon \rightharpoonup \hat{h}$ and the weak continuity of the bounded operator $T_{u_0}^2$, we obtain
	\begin{align*}
		\lim_{\varepsilon \to 0} \hat{h}_\varepsilon(0) 
		& = \ii \int_0^\infty \Big( \widehat{\Pi u_0} + z \hat{h} + 3t \widehat{T_{u_0}^2 h} \Big) \chi \, d\xi - \int_{0}^{\infty} \hat{h} \chi' d\xi  \\
		& = - \int_0^\infty (\partial_\xi \hat{h}) \chi \, d\xi - \int_{0}^{\infty} \hat{h} \chi' d\xi 
		= - \int_{0}^{\infty} \partial_\xi (\hat{h} \chi) d\xi = \hat{h}(0), 
	\end{align*} 
where we used the equation \eqref{eq:nonpertubedequation3} in Fourier space, $\partial_\xi \hat{h} = -i ( \widehat{\Pi u_0} + z \hat{h} + 3t \widehat{T_{u_0}^2 h} )$. This completes the proof. 

\end{proof}

Now, using the Lemma above, we present the proof of Proposition \ref{prop:BO3-weak}.
\begin{proof}[Proof of Proposition \ref{prop:BO3-weak}]
By $E_0$ conservation law, we have 
$$
\| u^\varepsilon (t, \cdot) \|_{L^2} =  \| u_{0}\|_{L^2}, \qquad \forall t \in \mathbb{R}. 
$$
Hence, for any fixed $t$, the family $\{u^\varepsilon(t)\}_{\varepsilon>0}$ is bounded in $L^2(\mathbb{R})$. 
By weak compactness, there exists a subsequence $\varepsilon_j$ tending to $0$ such that $u^{\varepsilon_j}(t)$ converges weakly in $L^2(\mathbb{R})$ to some limit $ZD[u_0](t, x)$. 
Since $u^\varepsilon$ is real valued, so is $ZD[u_0](t, x)$, hence, on the real line, we have
$$ZD[u_0](t, x) = \Pi ZD[u_0](t, x) + \overline{\Pi ZD[u_0](t, x)}.$$
By the definition of $I_+$, and the third point of Lemma \ref{lemma:perturbedbo3}, as $\varepsilon \to 0$, we have the convergence
$$ \Pi u^\varepsilon (t, z) = \frac{1}{2 \ii \pi} \, I_+ ( h_\varepsilon ) = \hat{h}_\varepsilon(0) \to \hat{h}(0) = \frac{1}{2 \ii \pi }\, I_+ \Bigl[(X^* - 3t\,T_{u_0}^2 - z\,\mathrm{Id})^{-1} (\Pi u_0)\Bigr], $$
where $h_\varepsilon$ and $h$ are the solutions to the equations \eqref{eq:pertubedequation3} and \eqref{eq:nonpertubedequation3}, respectively,  given in Lemma \ref{lemma:perturbedbo3}.
Therefore,
\begin{align}\label{eq:Pi-ZD}
	\forall z \in \mathbb{C}_+, \quad \Pi ZD[u_0](t, z) = \frac{1}{2 \ii \pi }\, I_+ \Bigl[(X^* - 3t\,T_{u_0}^2 - z\,\mathrm{Id})^{-1} (\Pi u_0)\Bigr].
\end{align}
We now prove the two additional properties of the map $u_0 \mapsto ZD[u_0]$.
The first property follows directly from the formula \eqref{eq:Pi-ZD}, so it remains to prove the second.
Since $u_0^\delta \to u_0$ strongly in $L^2(\mathbb R)$ and the sequence is uniformly bounded in $L^\infty$, we have 
$
T_{u_0^\delta} \xrightarrow[s]{} T_{u_0}
$. 
By uniform boundedness, the standard decomposition $T_{u_0^\delta}^2 - T_{u_0}^2 = T_{u_0^\delta} (T_{u_0^\delta} - T_{u_0} ) + (T_{u_0^\delta} -  T_{u_0} ) T_{u_0} $ implies
$$
T^2_{u_0^\delta} \xrightarrow[s]{\delta\to0}\; T^2_{u_0},$$
which results in
$$
(X^* - 3t\,T_{u_0^\delta}^2 - z\mathrm{Id})^{-1}
\;\xrightarrow[s]{\delta\to0}\;
(X^* - 3t\,T_{u_0}^2 - z\mathrm{Id})^{-1}.$$
Applying the formula \eqref{eq:Pi-ZD} again and using the strong
convergence of $\Pi u_0^\delta$ to $\Pi u_0$, we conclude that, for every $t\in\mathbb R$,
$$
ZD[u_0^\delta](t)
\rightharpoonup
ZD[u_0](t)
\quad \text{in } L^2(\mathbb R).$$
This completes the proof.
\end{proof}

We now complete the proof of Theorem \ref{thm: mainBO3eps}. By Proposition \ref{prop:BO3-weak}, it remains to identify the resolvent formula \eqref{eq:BO3T2}. Since the operator $T_{u_0}^2$ appears explicitly in that formula, it is necessary to obtain a precise formula for this
operator. We begin by treating the case where $u_0$ is a rational function with simple poles.

\begin{lemma}\label{lemma:Tu0Tu02}
Let $u_0(y)$ be a rational function with real coefficients of the form 
	$$
	u_0(y) = \sum_{j=1}^N \frac{c_j}{y - p_j} + \frac{\overline{c_j}}{y - \overline{p_j}}, 
	\qquad \textrm{Im} \, p_j > 0,
	$$
Then we have
	\begin{enumerate}
		\item $T_{u_0} f(y) = u_0(y) f(y) - \sum_{j=1}^N \frac{c_j\, f(p_j)}{y - p_j}, 
		$
		\item  $$
		\begin{aligned}
			T_{u_0}^2 f(y) 
			& =  u_0(y)^2 f(y) -  u_0(y)
		\sum_{j=1}^N \frac{c_j f(p_j)}{y - p_j} \\
			& - 
			\sum_{j=1}^N \frac{1}{y - p_j } \Bigg[ 
			c_j^2 f'(p_j) + \sum_{l \neq j} \frac{c_j c_l (f(p_l) - f(p_j)  )}{p_l - p_j} + c_j \sum_{l=1}^N \frac{\overline{c_l}\, f(p_j)}{p_j - \overline{p_l}} 
			\Bigg] 
		\end{aligned}.
		$$
	\end{enumerate}
\end{lemma}

\begin{proof}
Using a standard property of Toeplitz operators	with rational symbols, we have
	\begin{align}\label{eq:Tu0fy}
		T_{u_0} f(y) 
		 = \sum_{j=1}^N  c_j T_{ \frac{1}{y - p_j} } f(y)  + \sum_{j=1}^N  \bar c_j T_{ \frac{1}{y - \overline{p}_j}} f(y) 
		= \sum_{j=1}^N c_j \frac{f(y) - f(p_j) }{y - p_j} + \sum_{j=1}^N \frac{\overline{c_j}\, f(y)}{y - \overline{p}_j} .
	\end{align}
By the definition of $u_0(y)$, we obtain the first identity,
	$$T_{u_0} f(y) =  u_0(y) f(y) - \sum_{j=1}^N \frac{c_j\, f(p_j)}{y - p_j}.$$
To prove the second identity, we apply $T_{u_0}$ to \eqref{eq:Tu0fy},
\begin{align}\label{eq:Tu02fy}
		T^2_{u_0} f(y)   \nonumber
		 = T_{u_0} (T_{u_0} f)(y)  \nonumber	
		& =  \sum_{j=1}^N c_j \frac{T_{u_0} f(y) - T_{u_0} f(p_j) }{y - p_j} + \sum_{j=1}^N \frac{\overline{c_j}\, T_{u_0} f(y)}{y - \overline{p}_j}\\\nonumber
		& = u_0(y) T_{u_0} f(y) - \sum_{j=1}^N c_j \frac{ T_{u_0} f(p_j) }{y - p_j}\\
		& = u_0^2(y) f(y) - u_0(y) \sum_{j=1}^N \frac{c_j\, f(p_j)}{y - p_j} - \sum_{j=1}^N c_j \frac{ T_{u_0} f(p_j) }{y - p_j},
\end{align}
where we used the definition of $u_0(y)$ and the first identity. We observe that the right-hand side of \eqref{eq:Tu02fy} contains the singular term
$T_{u_0}f(p_j)$. To compute it, we use the representation \eqref{eq:Tu0fy}, and consider the sum
\begin{align*}
\sum_{l=1}^N c_l
\frac{ f(y)-f(p_l)}{y-p_l}
= c_j  \frac{f(y)-f(p_j)}{y-p_j} + \sum_{l\neq j}^N c_l
\frac{ f(y)-f(p_l)}{y-p_l}.
\end{align*}
For terms $l\neq j$, the denominator does not vanish at $y=p_j$, so we can substitute
$y=p_j$ directly. The term $l=j$ yields the contribution $c_j f'(p_j)$. Consequently, we obtain
\begin{align}\label{eq:Tu_0fpj}
T_{u_0} f(p_j) = c_j f'(p_j) + \sum_{l \neq j} \frac{c_l (f(p_l) - f(p_j))}{p_l - p_j} + \sum_{l=1}^N \frac{\overline{c_l}\, f(p_j)}{p_j - \overline{p_l}}.
\end{align}
Substituting this value into the formula \eqref{eq:Tu02fy}, we obtain the second identity. 
\end{proof}
It is now time to prove the zero dispersion limit geometric formula \eqref{eq:ZDu_0geom} in Theorem \ref{thm: mainBO3eps}.	
\begin{proof}[Proof of Theorem \ref{thm: mainBO3eps}]	
We first consider the case when  $u_0(y)$ is a rational function with real coefficients, with no pole on the real line, 
$$
u_0(y) = \frac{P_0(y)}{Q_0(y)},
$$
where, for some positive integer $N$, $P_0(y)$ is a real polynomial of degree at most $2N -1$, and $Q_0(y)$  is a monic real polynomial of degree $2N$.
With no loss of generality, we may assume that $Q_0$ has only
simple poles in the complex domain. Indeed, the rational function $u_0(y) = \frac{P_0(y)}{Q_0(y)}$ can be approximated in $L^2(\mathbb{R})$ by rational functions $u_0^\delta(y) = \frac{P_0(y)}{Q^\delta_0(y)}$ whose denominators have only simple zeros, obtained by slightly perturbing the multiple zeros of $Q_0 (y)$. Then the second property in Proposition \ref{prop:BO3-weak} ensures this assumption.  
We can then write
\begin{equation}\label{eq:initial-datum}
u_0(y)= \sum_{j=1}^{N} \frac{c_j}{y-p_j}
+\frac{\bar c_j}{\,y-\bar p_j\,}, \qquad \textrm{Im } p_j >0.	
\end{equation}
By the definition of $\Pi$, we have 
\begin{align}\label{eq:Pi_u_0}
\Pi u_0 (y) =  \sum_{j=1}^{N} \frac{\bar c_j}{y- \bar p_j} = u_0(y) - \sum_{j=1}^{N} \frac{c_j}{y-p_j}. 
\end{align}
Now, we consider the characteristic equation,  
$y - 3t\, u_0(y)^2 = z,$
which is equivalent to a polynomial equation of degree $4N+1$,
$$
y Q^2_0 (y) - 3t P_0^2(y) = z Q^2_0 (y).
$$
Observe that, for $z= x \in \mathbb{R}$, it has at least one real root and in fact an odd number $2\ell(t,x) +1 $ of real roots, counted with multiplicity. For a given $t$, except for a finite set $\Sigma_t$  of points $x$, we may consider the case where $\ell(t, x) = \ell \in \mathbb{Z}_{\geq 0}$, and where these roots are simple. Hence, for $x \in \mathbb{R} \setminus \Sigma_t$ there are exactly $2\ell+1$ real simple roots, which we label as
$$
y_0(t,x) < y_1(t,x) < \cdots < y_{2\ell}(t,x),$$
Moreover, the remaining $4N-2\ell$ solutions form complex conjugate pairs
$$
\text{Im } 	y_{2p-1}(t,x) >0, \quad y_{2p-1}(t,x) = \overline{y_{2p}(t,x)}, \qquad p=\ell+1,\dots,2N.$$
Next, let us define
$$
F(y;z) := y - 3t u_0(y)^2 - z.$$
Then, at any simple real root $y = y_k(t,x)$, $0 \leq k \leq 2N $, we have
$$
\partial_y F(y_k(t,x);x) = 1 - 6t\,u_0(y_k(t,x))u_0'(y_k(t,x)) \neq 0.
$$
By the continuity of $F(y;z)$, the sign of $\partial_y F(y_k(t,x);x)$ alternates at consecutive simple real zeros. Hence,
$$
\operatorname{sign}\bigl(1 - 6t\,u_0(y_k(t,x))u_0'(y_k(t,x))\bigr) = (-1)^k.$$
As all roots are simple, the implicit function theorem ensures that each root $y_k(t,x)$ extends to an analytic function $y_k(t,z)$ of $z$ in a neighborhood of the real axis. 
Differentiating the identity $F(y_k(t,z);z)=0$ with respect to $z$ yields
$$
\frac{\partial y_k}{\partial z}(t,z)
= \frac{1}{1 - 6t\,u_0(y_k(t,z))u_0'(y_k(t,z))}.
$$
In particular, at $z = x \in \mathbb{R}$, this derivative is real and has the sign $(-1)^k$.
Since $y_k(t, z)$ is analytic and real-valued for real $z$, the Cauchy--Riemann equations imply that, at $z = x \in \mathbb{R}$,
$$
\frac{\partial \text{Im } (y_k(t,z))}{\partial \text{Im } (z)}
=
\frac{\partial \text{Re } (y_k(t,z))}{\partial \text{Re } (z)}
= 
\text{Re }\left(\frac{\partial y_k}{\partial z}\right)
= 
\frac{1}{1 - 6t\,u_0(y_k(t,x))u_0'(y_k(t,x))},
$$
We observe that the variation of the imaginary part of $y_k$ under the vertical displacement of $z$ is governed by the sign of $\frac{\partial y_k}{\partial z}$, which is $(-1)^k$. 
Let $z = x + i\varepsilon$ with $\varepsilon > 0$ sufficiently small, we see that, the roots with even index lie in the upper half-plane, while those with odd index lie in the lower half-plane, namely,
\begin{equation*}
	\text{Im } (y_{k}(t,z))	
	\begin{cases}
		> 0, \qquad k \text{ even}\\
		< 0, \qquad k \text{ odd}.
	\end{cases}
\end{equation*}
For such a complex number $z$, the formula \eqref{eq:BO3T2} gives
$$
\Pi ZD [u_0](t,z) = \frac{1}{2 \ii \pi }\, I_+ \Bigl[(X^* - 3t\,T_{u_0}^2 - z\,\mathrm{Id})^{-1} \Pi u_0\Bigr].
$$
Let us define
	$
	f_{t,z} := (X^* - 3t\,T_{u_0}^2 - z\,\mathrm{Id})^{-1} \Pi u_0. 
	$
Hence,
	$$
	\Pi ZD [u_0] (t, z):= \frac{1}{2 \ii \pi }\, I_+ [f_{t,z}].
	$$
We shall show that $f_{t,z}$ is in fact a rational function. Recall that for any suitable function $f$, 
	$$
	X^* f(y) = y f(y) + \frac{1}{2 \ii \pi } I_+(f),
	$$
Applying this identity to $f_{t,z}$, we obtain
\begin{align*}
		\Pi u_0
		& = (X^* - 3t\,T_{u_0}^2 - z\,\mathrm{Id})\, f_{t,z}(y) \nonumber\\
		&= X^* f_{t,z}(y) - z\, f_{t,z} (y)- 3t\, T_{u_0}^2 f_{t,z}(y) \nonumber\\
		& = (y - z) f_{t,z}(y) + \frac{1}{2 \ii \pi } I_+(f_{t,z})
		- 3t\, T_{u_0}^2 f_{t,z}(y)\\
		& = (y - z) f_{t,z}(y) + \lambda (t,z)
		- 3t\, T_{u_0}^2 f_{t,z}(y).\nonumber
	\end{align*}
where $\lambda(t,z) := \frac{1}{2\pi \ii}\, I_+(f_{t,z})$. 
Applying the Lemma \ref{lemma:Tu0Tu02}, we have
	$$
	\Pi u_0 = (y - z - 3 t\, u_0^2(y)) f_{t,z}(y)  + \lambda(t,z) + u_0(y)
	\sum_{j=1}^N \frac{\mu_{j,1}}{y - p_j} +  \sum_{i=1}^N \frac{\mu_{j,2}}{y - p_j} ,
	$$
where 
	$$
	\mu_{j,1} = 3t c_j f_{t,z}(p_j), \quad
	\mu_{j,2} = 3t \Bigg[ 
	c_j^2 f'(p_j) + \sum_{l \neq j} \frac{c_j c_l (f(p_l) - f(p_j)  )}{p_l - p_j} + c_j \sum_{l=1}^N \frac{\overline{c_l}\, f(p_j)}{p_j - \overline{p_l}} 
	\Bigg].
	$$
Recall the relation \eqref{eq:Pi_u_0} between $u_0(y)$ and $\Pi u_0(y)$ , we  obtain
$$
	(y - z - 3 t\, u_0^2(y)) f_{t,z}(y)
	= -\lambda(t,z) + u_0(y)
	+ \sum_{j=1}^N  \frac{u_0(y) (- \mu_{j, 1} (t, z))}{y - p_j}
	+ \sum_{j=1}^N \frac{- \mu_{j,2}  (t, z)- c_j}{y - p_j}.
$$
Therefore, $f_{t,z}$ is rational,
$$
	f_{t,z} (y) = \frac{-\lambda(t,z) + u_0(y)
		+ \sum_{j=1}^N  \frac{u_0(y) (- \mu_{j,1} (t, z))}{y - p_j}
		+ \sum_{j=1}^N \frac{- \mu_{j, 3} (t, z)}{y - p_j}} {y-z - 3t u_0^2(y)},
$$
where $\mu_{j, 3} (t, z) = \mu_{j, 2} (t, z) + c_j$.
Let $y$ be any zero in the upper half-plane of the equation 
\begin{equation}\label{eq:equationcharac}
	y-z - 3t u_0^2(y) = 0.
\end{equation}
Evaluating the identity for $	f_{t,z}$ at these points eliminates the denominator and yields a linear relation between the unknown coefficients. Denoting these zeros by $\{y_{2k}\}_{k=0}^{2N}$, we obtain the following $(2N+1) \times (2N+1) $ linear system
	\begin{equation}\label{eq:lambdatz3}
		\lambda (t, z)
		+  \sum_{j=1}^N  \frac{u_0(y)  \mu_{j,1}(t, z)}{y_{2k} - p_j}
		+ \sum_{j=1}^N \frac{\mu_{j,3}(t, z)}{y_{2k} - p_j}
		=  u_0 (y_{2k}),\; j = 0, 1, \cdots, 2N.
	\end{equation}
The Cramer's rule thus gives
\begin{equation}\label{eq:lambdaND}
\lambda (t,z) = \frac{\mathcal{N}}{\mathcal{D}},
\end{equation}
where $\mathcal{D}$ denotes the determinant of the coefficient matrix, 
\begin{equation}\label{eq:determinant_D}
\mathcal{D} =
\begin{vmatrix}
	\dfrac{1}{y_0 - p_1} & \cdots & \dfrac{1}{y_0 - p_N}   & \dfrac{u_0(y_0)}{y_0 - p_1} & \cdots &  \dfrac{u_0(y_0)}{y_0 - p_N}  & 	1  \\[6pt]
	\dfrac{1}{y_2 - p_1} & \cdots & \dfrac{1}{y_2 - p_N}   & \dfrac{u_0(y_2)}{y_2 -p_1} & \cdots &  \dfrac{u_0(y_2)}{y_2 - p_N} & 	1  \\[6pt]
	\cdots \\[6pt]
	\dfrac{1}{y_{4N} - p_1} & \cdots & \dfrac{1}{y_{4N} - p_N}  & \dfrac{u_0(y_{4N})}{y_{4N} - p_1} & \cdots &  \dfrac{u_0(y_{4N})}{y_{4N} - p_N} & 1
\end{vmatrix},
\end{equation}
and $\mathcal{N}$ is obtained by replacing its last column with the right-hand side vector,
\begin{equation}\label{eq:determinant_N}
\mathcal{N} =
\begin{vmatrix}
	\dfrac{1}{y_0 - p_1} & \cdots & \dfrac{1}{y_0 - p_N}   & \dfrac{u_0(y_0)}{y_0 - p_1} & \cdots &  \dfrac{u_0(y_0)}{y_0 - p_N}  & 	u_0(y_0) \\[6pt]
	\dfrac{1}{y_2 - p_1} & \cdots & \dfrac{1}{y_2 - p_N}   & \dfrac{u_0(y_2)}{y_2 -p_1} & \cdots &  \dfrac{u_0(y_2)}{y_2 - p_N} & 	u_0(y_2)  \\[6pt]
	\cdots \\[6pt]
	\dfrac{1}{y_{4N} - p_1} & \cdots & \dfrac{1}{y_{4N} - p_N}  & \dfrac{u_0(y_{4N})}{y_{4N} - p_1} & \cdots &  \dfrac{u_0(y_{4N})}{y_{4N} - p_N} & 	u_0(y_{4N}) 
\end{vmatrix}.
\end{equation}

Before continuing the proof, let us state the following proposition in order to reduce the two determinants above to Vandermonde-type determinants.
\begin{proposition}\label{prop:ND_n=2}
Let $u_0$ be as defined in Theorem \ref{thm: mainBO3eps}. For a fix $z \in \mathbb{C}$, let $\{y_{2k}\}_{k=0}^{2N}$ be the zeros of the characteristic equation \eqref{eq:equationcharac} located in the upper half-plane. 

Then, there exists a constant $K$, depending on $N,t, p_1,\dots,p_N$,  the determinants $\mathcal{D}$ and  $\mathcal{N}$ given in \eqref{eq:determinant_D} and \eqref{eq:determinant_N}, respectively, can be reduced to the following determinants, respectively,
\begin{align*}
& \mathcal{D} =
\frac{K}{\prod_{k=0}^{2N} \prod_{\nu=1}^N ( 3t v_k^2 + z - p_\nu)}
\det \mathcal{V}(v_0, \cdots, v_{2N})
, \\
&	\mathcal{N}
=
\frac{K}{\prod_{k=0}^{2N} \prod_{\nu=1}^N ( 3t v_k^2 + z - p_\nu)} \det \widetilde{\mathcal{V}}(v_0, \cdots, v_{2N}),
\end{align*}
where $v_k : = u_0(y_{2k})$ for $ 0 \leq k \leq 2N$, and $\mathcal{V}$
is the standard Vandermonde matrix and
$\widetilde{\mathcal{V}}$ a missing power Vandermonde 
\begin{align}\label{eq:vander}
\mathcal{V}
= 
\begin{pmatrix}
	1 & v_{k} & \dots &  v_{k}^{2N-1} &  v_{k} ^{2N} \\
\end{pmatrix}_{k=0}^{2N},\quad
\widetilde{\mathcal{V}}
= 
\begin{pmatrix}
	1 & v_{k} & \dots &  v_{k}^{2N-1} &  v_{k} ^{2N+1} \\
\end{pmatrix}_{k=0}^{2N}.
\end{align}
\end{proposition}

\begin{proof}
We first treat the determinant $\mathcal{D}$. 
By the definition of $y_{2k}$, we have,  for $ 0 \leq k \leq 2N$, 
$$
	\frac{1}{y_{2k} - p_\nu}
	= 
	\frac{1}{3 t v_k^2 + z - p_\nu}, \quad
	\frac{v_k}{y_{2k} - p_\nu}
	= 
	\frac{v_k}{3 t v_k^2 + z - p_\nu}, \quad \nu=1, \cdots, N.$$
Hence the $k$-th row of the determinant $\mathcal{D}$ can be written in the form
\begin{align*}
\Big(
\frac{1}{3 t v_k^2 + z - p_1}, \cdots, \frac{1}{3 t v_k^2 + z - p_N},\;
\frac{v_k}{3 t v_k^2 + z - p_1}, \cdots, \frac{v_k}{3 t v_k^2 + z - p_N}, 1 \Big).
\end{align*}
Define the polynomials (in the formal variable $v$)
$$
Q(v) := \prod_{\nu=1}^N (3t v^2 + z - p_\nu),
\qquad
P_{\nu}(v) :=  \prod_{\mu\neq\nu}^N (3t v^2 + z - p_\mu).$$
Later these polynomials will be evaluated at $v = v_k$.

Multiplying the $k$-th row by $Q(v_k)$, we obtain
$$
\mathcal D
=
\frac{1}{\prod_{k=0}^{2N}  Q(v_k)}
\det
\Big(
P_1(v_k),\dots,P_N(v_k),\;
v_kP_1(v_k),\dots,v_k P_N(v_k),\;
Q(v_k)
\Big)_{k=0}^{2N}.$$
We observe that 
\begin{align*}
\deg Q(v) = 2N, \qquad \deg \{ P_{\nu} (v), v P_{\nu} (v) \} \leq  2N -1.
\end{align*}
so these polynomials have degree at most $2N$, hence belong to $\mathbb{C}[v]_{\le 2N}$,
which is a vector space of dimension $2N+1$. Since the polynomials $\{P_1,\dots,P_N,\;
v P_1,\dots,v P_N,\;
Q \}$ are linearly independent,  they form a basis of the vector space  $\mathbb{C}[v]_{\le 2N}$.

Expanding these polynomials in the monomial basis $(1,v,\dots,v^{2N})$, we obtain a factorization
$$
\Big(
P_1(v_k),\dots,P_N(v_k),\;
v_kP_1(v_k),\dots,v_k P_N(v_k),\;
Q(v_k)
\Big)_{k=0}^{2N} = \mathcal V(v_0,\dots,v_{2N}) \cdot \mathrm{Coe}_D,
$$
where $\mathcal V$ is the Vandermonde matrix and $\mathrm{Coe}_D$ is the coefficient matrix. Therefore
$$
\mathcal D
=
\left( \frac{1}{ \prod_{k=0}^{2N} Q(v_k)} \right)
\det \mathcal V(v_0,\dots,v_{2N}) \cdot \det(\mathrm{Coe_D}).
$$
Next, we study the structure of $\mathrm{Coe_D}$. 
Expanding all polynomials $P_{\nu}(v)$,
\begin{align*}\label{eq:Pnu}
	P_{\nu}(v) =  \prod_{\mu\neq\nu}^N (3t v^2 + z - p_\mu) = \sum_{r=0}^{N-1} A_{r,\nu} v^{2r}, \quad \text{with} \quad A_{r,\nu} = e^{(\nu)}_{N-1-r} (3 t)^r.
\end{align*}
where $e^{(\nu)}_{N-1-r}$ is precisely the symmetric polynomial, and the superscript $(\nu)$ means the removal of the $\nu$-th one.
Observe that $P_\nu(v)$ and $Q(v)$ contain only even powers of $v$, whereas $vP_\nu(v)$ contains only odd powers. 
We reorder the monomial basis as
$$
(1, v^2, \dots, v^{2N-2} \mid v, v^3, \dots, v^{2N-1} \mid v^{2N}),
$$
and order the columns as
$$
(P_1,\dots,P_N \mid vP_1,\dots,vP_N \mid Q),
$$
then the coefficient matrix $\mathrm{Coe}_{new}$ becomes the block upper triangular form
$$
\mathrm{Coe}_{D, new} =
\begin{pmatrix}
	A & 0 & * \\
	0 & A & 0 \\
	0 & 0 & (3t)^N 
\end{pmatrix}, \quad \text{with} \quad
A = (A_{r,\nu})_{r = 0, \cdots, N-1}^{\nu = 1, \cdots, N}.
$$
Reordering the monomial basis introduces a factor $(-1)^{\frac{N(N-1)}{2}}$, hence
$$
\det(\mathrm{Coe}_D) = (-1)^{\frac{N(N-1)}{2}} (3t)^N  (\det A)^2 : = K, 
$$
which proves the reduction formula for $\mathcal{D}$. 

We now turn to the determinant $\mathcal{N}$. Proceeding in exactly the same way, we obtain
$$
\mathcal N
=
\frac{1}{\prod_{k=0}^{2N}  Q(v_k)}
\det
\Big(
P_1(v_k),\dots,P_N(v_k),\;
v_kP_1(v_k),\dots,v_k P_N(v_k),\;
v_kQ(v_k)
\Big)_{k=0}^{2N}.$$
The only difference from the previous case is that the last polynomial is now $v Q(v)$, which has degree $2N+1$. All other polynomials have degree at most $2N -1$.  
In particular, no polynomial contains a $v^{2N}$ term. Hence the set of monomials involved is
$\{1,v,\dots,v^{2N-1},v^{2N+1}\},$ which produces the missing-power Vandermonde matrix.
Repeating the same coefficient-matrix factorization, we obtain
$$
\mathcal N
=
\left( \frac{1}{ \prod_{k=0}^{2N} Q(v_k)} \right)
\det \widetilde{\mathcal{V}}(v_0,\dots,v_{2N}) \cdot \det(\mathrm{Coe}_N).
$$
where
$$
\widetilde{\mathcal V}
=
\big( 1, v_k, \dots, v_k^{2N-1}, v_k^{2N+1} \big)_{k=0}^{2N}
$$
After reordering the monomial basis and
the polynomial columns as before, the coefficient matrix $\mathrm{Coe}_N$ again becomes block
upper triangular. Indeed, the replacement of $Q(v)$ by $vQ(v)$ only shifts the highest-degree monomial from $v^{2N}$ to $v^{2N+1}$, preserves  its coefficient $(3t)^N$, and does not affect the lower-degree block structure. Consequently, 
$$\det (\mathrm{Coe}_N) = \det(\mathrm{Coe}_D) = K. $$
This proves the reduction formula for 
$\mathcal{N}$ and completes the proof.
\end{proof}

As established in Proposition above, both $\mathcal{N}$ and $\mathcal{D}$ reduce to Vandermonde-type determinants multiplied by the same factor. Consequently, according to \eqref{eq:lambdaND}, the expression for $\lambda (t,z)$ simplifies to the ratio of the missing power Vandermonde determinant $\det \widetilde{\mathcal{V}}$ and the standard Vandermonde determinant $\det \mathcal{V}$. We give a more precise characterization of this ratio in the following lemma.
\begin{lemma}\label{lemma:ratio-Vander}
In the conditions and notations of Proposition \ref{prop:ND_n=2}, 	
the following identity holds:
	\begin{equation*}
		\det \widetilde{\mathcal{V}} = \sum_{k=0}^{2N} v_k \det \mathcal{V}.
	\end{equation*}
where $\mathcal{V}$ and $\widetilde{\mathcal{V}}$ are given in \eqref{eq:vander}.
\end{lemma}
\begin{proof}
Consider the linear mapping $\Lambda : \mathbb{C}_{\leq 2N+1}[x] \to \mathbb{C}$ defined by the determinant of the matrix where the last column is the evaluation of a polynomial $\mathcal{P}$ at the points $v_k$:
\begin{equation*}
		\Lambda(\mathcal{P}) = \det \begin{pmatrix} 1 & v_{k} & \dots & v_{k}^{2N-1} & \mathcal{P}(v_k) \end{pmatrix}_{k = 0}^{2N}.
\end{equation*}
First, we observe that if $\text{deg}(\mathcal{P}) \leq 2N-1$, the last column $\mathcal{P}(V)$ is a linear combination of the first $2N$ columns. By the properties of the determinant, we have
\begin{equation}\label{eq:Lambda_low}
	\Lambda(\mathcal{P}) = 0 \quad \text{for any } \mathcal{P} \text{ such that } \text{deg}(\mathcal{P}) \leq 2N-1.
\end{equation}
Let $\mathcal{Q}(x)$ be the monic polynomial of degree $2N+1$,
\begin{equation*}
\mathcal{Q}(x) = \prod_{k=0}^{2N} (x - v_k) = x^{2N+1} - \sum_{k=0}^{2N} v_k  x^{2N} + \mathcal{Q}_{low}(x),
\end{equation*}
where $\mathcal{Q}_{low}(x)$ is a polynomial of degree at most $2N-1$. Thus,  $\Lambda(\mathcal{Q}_{low}) = 0$. Since $\mathcal{Q}(v_k) = 0$ for all $k = 0, \dots, 2N$, by the definition of $\Lambda$, we obtain
\begin{equation*}
	\Lambda(\mathcal{Q}) = \det \begin{pmatrix} 1 & v_{k} & \dots & v_{k}^{2N-1} & \mathcal{Q}(v_k) \end{pmatrix}_{k = 0}^{2N} = 0.	
\end{equation*}
Using the linearity of $\Lambda$ and the expansion of $\mathcal{Q}(x)$, we have:
\begin{align*}
0 = \Lambda(\mathcal{Q}) = \Lambda\Big( x^{2N+1} - \sum_{k=0}^{2N} v_k \, x^{2N} + \mathcal{Q}_{low}(x) \Big) = \Lambda(x^{2N+1}) - \sum_{k=0}^{2N} v_k \, \Lambda(x^{2N}) + \Lambda(\mathcal{Q}_{low}).
\end{align*}
Therefore, we obtain:
\begin{equation*}
\Lambda(x^{2N+1}) = \sum_{k=0}^{2N} v_k  \Lambda(x^{2N}).
\end{equation*}
Recognizing that $\Lambda(x^{2N+1}) = \det \widetilde{\mathcal{V}}$ and $\Lambda(x^{2N}) = \det \mathcal{V}$ completes the proof.
\end{proof}

Now, we return to the proof of Theorem \ref{thm: mainBO3eps}. Applying the Proposition \ref{prop:ND_n=2} and Lemma \ref{lemma:ratio-Vander}, we obtain
$$
\lambda(t, x) = \frac{\mathcal{N}}{\mathcal{D}} = \frac{\det \widetilde{\mathcal{V}}(v_0, \cdots, v_{2N})}{\det \mathcal{V}(v_0, \cdots, v_{2N})} =\sum_{k=0}^{2N} v_k = \sum_{k=0}^{2N}u_0(y_{2k}),
$$
where $v_k : = u_0(y_{2k})$ for $ 0 \leq k \leq 2N$, $\mathcal{V}$ and 
$\widetilde{\mathcal{V}}$ are given in \eqref{eq:vander}.
We already know that, the characteristic equation $y - 3t\, u_0(y)^2 = z$ admits $2\ell+1$ of real solutions
$
y_0(t,x) < y_1(t,x) < \cdots < y_{2\ell}(t,x),
$
and the remaining $2(2N-\ell)$ solutions form complex conjugate pairs 
$ y_{2p-1}(t,x) = \overline{y_{2p}(t,x)}, p=\ell+1,\dots,2N,$ with $\text{Im } y_{2p-1}(t,x) >0$.
By $
ZD [u_0] (t, x) = \Pi ZD [u_0] (t, x) + \overline{ \Pi ZD [u_0] (t, x)} =\lambda(t, x)+ \overline{ \lambda(t, x)}$, we thus have
\begin{align*}
ZD [u_0] (t, x)
& =  \sum_{k=0}^{2N}u_0(y_{2k}(t, x)) + \sum_{k=0}^{2N}u_0( \overline{y_{2k}(t, x)})\\ 
& = 2 \sum_{k=0}^{\ell}u_0(y_{2k}(t, x)) +  \sum_{k=\ell+1}^{2N}u_0(y_{2k}(t, x))  +  \sum_{k=\ell+1}^{2N}u_0( \overline{y_{2k}(t, x)}) \\
& = 2 \sum_{k=0}^{\ell}u_0(y_{2k}(t, x)) + \sum_{ k =2 \ell+1}^{4N}u_0(y_{k}(t, x)) 
\\
& = \sum_{\substack{0 \le k \le 4N \\ y_k \in \mathbb{R}}} (-1)^k\, u_0(y_k),
\end{align*}
where we used Lemma \ref{lemma:fracPQn=2} (see Appendix) in the last line.
This finishes the proof of the formula \eqref{eq:ZDu_0geom} for the case when $u_0(y)$ is rational. 

We now extend the formula \eqref{eq:ZDu_0geom} for rational initial data to a general function 
$u_0 \in L^2(\mathbb{R}) \cap C^1(\mathbb{R})$ satisfying
$|u_0(y)| + |u_0'(y)| \to 0$, as $|y|\to\infty.$
Using a standard mollification argument, we construct a sequence 
$u_0^\delta \in H^s(\mathbb{R})$ for every $s\in\mathbb{R}$ such that $u_0^\delta \to u_0$ strongly in $L^2(\mathbb{R}),$
with a uniform bound in $L^\infty(\mathbb{R})$. Since rational functions are dense in 
$H^2(\mathbb{R})$, we may assume that each $u_0^\delta$ is rational.

Let us fix $t\in\mathbb{R}$ and define
$
g_t(y)=y-3t u_0^2(y)$. Denote by $K_t(u_0)$ the set of critical values of $g_t$. 
Since $u_0(y)\to0$ and $u_0'(y)\to0$ as $|y|\to\infty$, we have
$g_t'(y) = 1-6t u_0(y)u_0'(y)\to1$, so $K_t(u_0)$ is compact. 
By Sard's theorem,
$K_t(u_0)$ has Lebesgue measure zero.
Let $\Omega$ be any connected component of $K_t(u_0)^c$. For every $x_0\in \Omega$, the equation
$
y-3t u_0^2(y)=x_0$
has only simple roots. Consequently, there exists an open subinterval $\omega$ of $\Omega$, 
containing $x_0$ such that, for every $x\in\omega$, the equation
admits the same finite number $2\ell+1$ of real solutions
\[
y_0(t,x)<\cdots<y_{2\ell}(t,x).
\]
Since $u_0^\delta \to u_0$ in $C^1_{\mathrm{loc}}(\mathbb{R})$, we have 
$g_t^\delta(y)=y-3t (u_0^\delta(y) )^2 \to g_t$ in $C^1_{\mathrm{loc}}$, and the simple roots
of $g_t(y)=x$ are stable under such perturbations. Hence, for $\delta$ sufficiently small, we have $\omega \cap K_t(u_0^\delta) = \emptyset$, and every $x\in\omega$, the equation
$
y-3t (u_0^\delta(y))^2=x $
has exactly $2l +1$ simple solutions
$$
y_0^\delta(t,x)<\cdots<y_{2 l}^\delta(t,x),
$$
which satisfy
$y_k^\delta(t,x)\to y_k(t,x)$ as $\delta\to0$ for every $k=0,\dots,2 l$.

We have already proved that, for $x\in\omega$,
\[
ZD[u_0^\delta](t,x)
=
\sum_{k=0}^{2\ell} (-1)^k
u_0^\delta\!\left(y_k^\delta(t,x)\right).
\]
Using the strong $L^2$ convergence of $u_0^\delta$, the uniform $L^\infty$
bound, we pass to the weak limit in
$L^2(\omega)$. Therefore, for almost
every $x\in\omega$,
$$
ZD[u_0](t,x)
=
\sum_{k=0}^{2\ell} (-1)^k
u_0\!\left(y_k(t,x)\right).
$$
Since $\omega$ is an arbitrary relatively compact subinterval of $\Omega$, the above identity holds for almost
every $x\in K_t(u_0)^c$. This completes the whole proof.

\end{proof}

\section{Zero dispersion limit for the BO hierarchy}\label{seczeroBon}
In this section, we analyze the zero-dispersion limit for the $(n+1)$-th order flow of the BO hierarchy. We consider the equation with a small dispersion parameter $\varepsilon > 0$
\begin{align}\label{BOhiereps}
	\partial_t u^\varepsilon = \partial_x (\nabla E_n^\varepsilon (u^\varepsilon)),  \qquad
	u^\varepsilon(0, x) = u_0(x) \in L^2_r (\mathbb{R}),
\end{align}
where the Hamiltonian is defined as
$$E_n^\varepsilon (u) := \langle (\varepsilon D - T_{u})^n \Pi u, \Pi u \rangle_{L^2}.$$
Following the approach used in Section \ref{seczeroBo3}, we introduce the rescaling transformation:
\begin{equation*}
	u^\varepsilon(t, x) = \varepsilon\, v(s, x), \quad \text{with} \quad s = \varepsilon^n t.
\end{equation*}
Under this change of variables, the function $v(s, x)$ satisfies the $(n+1)$-th order flow,
\begin{equation}\label{eq:BOhier_v}
	\partial_s v = \partial_x (\nabla E_n (v)), \quad
	v(0, x) = v_0(x) := \varepsilon^{-1} u_{0}(x),
\end{equation}
where $E_n(v) = \langle (D - T_v)^n \Pi v, \Pi v \rangle_{L^2}$.
From Section \ref{secproof}, we know that, if $v_0 \in H^{s}_r (\mathbb{R})$, $s > n + \frac{1}{2}$, the solution of the initial value problem \eqref{eq:BOhier_v} is given by 
$$
v(s,x) = \Pi v(s,x) + \overline{\Pi v(s,x)}, 
\qquad x \in \mathbb{R},
$$
where for every complex number $z$ with $\mathrm{Im}(z)>0$,
\begin{align}\label{formula_BO3}
	\Pi v(s, z) 
	&= \frac{1}{2i \pi} \, I_+ \left[ \left( X^* - (n+1)s L_{v_0}^n - z \mathrm{Id} \right)^{-1} \Pi v_0 \right].	
\end{align}
We know from Remark \ref{remark:well-def} that, the resolvent on the right hand side is well defined.
Recall $L_{v_0} = D - T_{v_0}$, using a binomial expansion, we can expand $L_{v_0}^n$ as 
\begin{equation*}
 L_{v_0}^n = 	( D - T_{u_0})^n = (-1)^n T_{u_0}^n +  \sum_{j=1}^n R_j,
\end{equation*}
where each $R_j$ is a combination of operators $\{D^k\}_{k=1}^{n}$ and $\{T_{u_0}^m\}_{m=1}^{n-1}$, such that the total differential order of each term is at most $j$.

Moreover, scaling the solution back via $\varepsilon \Pi v(s,x) = \Pi u^\varepsilon(t,x)$ and $\varepsilon \Pi v_0 = \Pi u_0$, we can write the solution to the original initial value problem \eqref{BOhiereps} as
$$
u^\varepsilon(t,x) = \Pi u^\varepsilon(t,x) + \overline{\Pi u^\varepsilon(t,x)}, 
\qquad x \in \mathbb{R},$$
where  for every complex number $z$ with $\mathrm{Im}(z)>0$,
\begin{align*}	
	\Pi u^\varepsilon(t, z)  
	& = \frac{1}{2i \pi} \, I_+ \left[ \left( X^* - (n+1)t (\varepsilon D-T_{u_0})^n - z \mathrm{Id} \right)^{-1} \Pi u_0 \right].
\end{align*}
By scaling and Remark \ref{remark:well-def}, we have
$$e^{-i \frac{t}{\varepsilon} (\varepsilon D - T_{u_0} )^{n+1}}  U^*(t) X^* U(t)  e^{i \frac{t}{\varepsilon} (\varepsilon D - T_{u_0} )^{n+1}}
= X^* - (n+1) t (\varepsilon D - T_{u_0})^n.$$
which is maximally dissipative and guarantees that the resolvent on the right hand side of $\Pi u^\varepsilon(t, z)$ is well defined. 

The following Lemma is a general version of Lemma  \ref{lemma:perturbedbo3}.
\begin{lemma}\label{lemma:perturbedbo_n}
Let $z \in \mathbb{C}$ with $\mathrm{Im}(z) > 0$, and assume $u_0 \in H^{s}_r (\mathbb{R})$, $s > n + 1/2$. Let $h_\varepsilon, h \in L^2(\mathbb{R_+})$ be the solutions to the two equations below, respectively, 
	\begin{align}
		& \left( X^* - (n+1)t (\varepsilon D-T_{u_0})^n  - z \mathrm{Id} \right) h_\varepsilon =\Pi u_0, \label{eq:pertubedequation}\\
		& \left(X^* - (-1)^{n}(n+1) t\,T_{u_0}^n - z\,\mathrm{Id}\right)h =  \Pi u_0\label{eq:nonpertubedequation}.
	\end{align}
Then, as $\varepsilon \to 0$, we have
	\begin{enumerate}
		\item  $h_\varepsilon$ is uniformly bounded in $L^2(\mathbb{R}_+)$;
		\item $\hat{h}_\varepsilon \rightharpoonup \hat{h}$ weakly in $L^2(\mathbb{R}_+)$;
		\item  $\hat{h}_\varepsilon(0) \to \hat{h}(0)$.
	\end{enumerate}
\end{lemma}

\begin{proof}
The maximal dissipativity of the operator $X^* - (n+1)t (\varepsilon D - T_{u_0})^n$ allows us to obtain the uniform $L^2$ bound for $h_\varepsilon$. Thus, there exists a subsequence such that $\hat{h}_\varepsilon \rightharpoonup \hat{g}$ weakly in $L^2(\mathbb{R}_+)$. To identify the limit, we examine the equation in Fourier space. 
\begin{equation}\label{eq:pertueqn}
\partial_\xi \hat{h}_\varepsilon = -i ( \widehat{\Pi u_0} + z \hat{h}_\varepsilon + (n+1)t
\mathcal{F}[(\varepsilon D - T_{u_0})^n h_\varepsilon] ).
\end{equation}
For any test function $\phi \in C_c^\infty((0, \infty))$, we have:
\begin{align}\label{eq:fouriern}
\int_0^\infty \hat{h}_\varepsilon (-i \partial_\xi \phi - z \phi) d\xi - (n+1)t \int_0^\infty \mathcal{F}[(\varepsilon D - T_{u_0})^n h_\varepsilon] \phi \, d\xi = \int_0^\infty \widehat{\Pi u_0} \phi \, d\xi.
\end{align}
The operator $(\varepsilon D - T_{u_0})^n$ can be expanded as,
\begin{equation}\label{eq:expansion}
	(n+1)t (\varepsilon D - T_{u_0})^n = (-1)^n (n+1)t T_{u_0}^n + (n+1)t \varepsilon \mathcal{R}_\varepsilon,
\end{equation}
where $\mathcal{R}_\varepsilon = \sum_{j=1}^n \varepsilon^{j-1} R_j$ and each $R_j$ is a combination of operators $\{D^k\}_{k=1}^{n}$ and $\{T_{u_0}^m\}_{m=1}^{n-1}$, such that the total differential order of each term is at most $j$.
Using this expansion, we split the second integral in the left hand-side of \eqref{eq:fouriern},
$$ (n+1)t \int_0^\infty \mathcal{F}[(\varepsilon D - T_{u_0})^n h_\varepsilon] \phi \, d\xi = (-1)^n(n+1)t \int_0^\infty \widehat{T_{u_0}^n h_\varepsilon} \phi \, d\xi + (n+1)t \varepsilon \int_0^\infty \widehat{\mathcal{R}_\varepsilon h_\varepsilon} \phi \, d\xi. $$
For any combination of operators in $R_j$ (e.g., $T_{u_0} D^k T_{u_0}$), because the test function $\phi$ has compact support, any term like $\xi^k \phi(\xi)$ in the integral $\int \mathcal{F}[R_j h_\varepsilon] \phi \, d\xi$, is a bounded function with finite $L^2$ norm. 
Since the Toeplitz operators $T_{u_0}$ are bounded on $L^2$ and $h_\varepsilon$ is uniformly bounded, the entire integral $\int \widehat{R_j h_\varepsilon} \phi \, d\xi$ is bounded by a constant independent of $\varepsilon$. Therefore,
$$ \left| \varepsilon \int_0^\infty \widehat{\mathcal{R}_\varepsilon h_\varepsilon} \phi \, d\xi \right| \leq \varepsilon \sum_{j=1}^n \varepsilon^{j-1} C_j \xrightarrow{\varepsilon \to 0} 0. $$
Passing to the limit in \eqref{eq:fouriern}, we find that $\hat{g}$ satisfies equation \eqref{eq:nonpertubedequation} in the sense of distributions, implying $g=h$.
	
For the third point, let $\chi \in C_c^\infty([0, \infty))$ with $\chi(0)=1$. Expressing $\hat{h}_\varepsilon(0)$ via integration by parts and substituting $\partial_\xi \hat{h}_\varepsilon$ from equation \eqref{eq:pertueqn},
	$$ \hat{h}_\varepsilon(0) = i \int_0^\infty \left( \widehat{\Pi u_0} + z \hat{h}_\varepsilon + (n+1)t \mathcal{F}[(\varepsilon D - T_{u_0})^n h_\varepsilon] \right) \chi \, d\xi - \int_0^\infty \hat{h}_\varepsilon \chi' \, d\xi. $$
Applying the expansion \eqref{eq:expansion} again, we isolate the $\varepsilon$-dependent terms as follows:
\begin{align*}
	\hat{h}_\varepsilon(0) &= i \int_0^\infty \left( \widehat{\Pi u_0} + z \hat{h}_\varepsilon + (-1)^n (n+1)t \widehat{T_{u_0}^n h_\varepsilon} \right) \chi \, d\xi - \int_0^\infty \hat{h}_\varepsilon \chi' \, d\xi + I_\varepsilon,
\end{align*}
where the error term is given by
$ I_\varepsilon = i(n+1)t \varepsilon \int_0^\infty \widehat{\mathcal{R}_\varepsilon h_\varepsilon}(\xi) \chi(\xi) \, d\xi. $
Using the same argument above, $I_\varepsilon$ vanishes as $\varepsilon \to 0$. Taking the limit $\varepsilon \to 0$, the weak convergence $\hat{h}_\varepsilon \rightharpoonup \hat{h}$ allows us to pass to the limit in the integrals,
	$$ \lim_{\varepsilon \to 0} \hat{h}_\varepsilon(0) = - \int_0^\infty (\partial_\xi \hat{h}) \chi \, d\xi - \int_0^\infty \hat{h} \chi' \, d\xi = - \int_0^\infty \partial_\xi(\hat{h}\chi) \, d\xi = \hat{h}(0). $$
This concludes the proof for the general case.
\end{proof}

As in the third-order case, we first state the weak convergence result and the resolvent formula for the general hierarchy. 

\begin{proposition}\label{prop:BOhier-weak}
		Let $n \in \N$ and $u_0 \in H^s_r(\mathbb{R})$ with $s > n + \frac{1}{2}$. For every $t \in \mathbb{R}$, the solution $u^\varepsilon(t)$ of \eqref{eq:BOneps} converges weakly in $L^2(\mathbb{R})$ to a function $ZD[u_0](t, \cdot)$, characterized by
		$$
		\forall x \in \mathbb{R}, \quad ZD[u_0](t, x) = \Pi ZD[u_0](t, x) + \overline{\Pi ZD[u_0](t, x)},
		$$
		where, for every complex number $z$ with $\textrm{Im} \,z >0$, 
		\begin{align}\label{eq:BOnTn}
			\Pi ZD[u_0](t, z) = \frac{1}{2 \ii \pi }\, I_+ \Bigl[(X^* - (-1)^n	(n+1)t\,T_{u_0}^n - z\,\mathrm{Id})^{-1} (\Pi u_0)\Bigr].
		\end{align}
		Moreover, the mapping
		$$	u_0 \in L^\infty(\mathbb{R}) \cap L^2(\mathbb{R}) \longmapsto ZD[u_0]$$
		has the following properties:
		\begin{enumerate}
			\item For every $t \in \mathbb{R}$,
			$
			\|ZD[u_0](t)\|_{L^2} \le \|u_0\|_{L^2},$
			and $ZD[u_0]$ is continuous on $\mathbb{R}$ with values in $L^2(\mathbb{R})$ endowed with the weak topology.
			\item If $u_0^\delta$ converges strongly to $u_0$ in $L^2(\mathbb{R})$ as $\delta \to 0$, with a uniform bound on $\|u_0^\delta\|_{L^\infty}$, then, for every $t \in \mathbb{R}$, $ZD[u_0^\delta](t)$ converges to $ZD[u_0](t)$ weakly in $L^2(\mathbb{R})$.
		\end{enumerate}	
\end{proposition}

\begin{proof}[Proof of Propostion \ref{prop:BOhier-weak}]
The proof follows the same strategy as the third-order case, substituting the operator $(X^* - 3t T_{u_0}^2 - z\text{Id})$ with the $n$-th order version. By the definition of $I_+$, and applying the limit result in Lemma \ref{lemma:perturbedbo_n}, the solution $\Pi u^\varepsilon$ converges as $\varepsilon \to 0$ to
\begin{align}\label{eq:Pi-ZD-n}
	\Pi ZD[u_0](t, z) = \frac{1}{2 \ii \pi }\, I_+ \Bigl[\left(X^* - (-1)^{n}(n+1) t\,T_{u_0}^n - z\,\mathrm{Id}\right)^{-1} \Pi u_0 \Bigr].
\end{align}
By the conservation law and weak compactness, the solution $u^{\varepsilon}(t)$ converges weakly in $L^2(\mathbb{R})$ to a function $ZD[u_0](t, x)$, characterized by
$$ZD[u_0](t, x) = \Pi ZD[u_0](t, x) + \overline{\Pi ZD[u_0](t, x)}.$$
where $	\Pi ZD[u_0](t, z)$ is given in \eqref{eq:Pi-ZD-n}.

The two additional properties of the map $u_0 \mapsto ZD[u_0]$ are proved by adapting the third-order argument to the $n$-th order operator. 
We utilize the telescoping sum identity
\begin{align*}
	T_{u_0^\delta}^n - T_{u_0}^n = \sum_{j=0}^{n-1} T_{u_0^\delta}^{n-1-j} (T_{u_0^\delta} - T_{u_0}) T_{u_0}^j.
\end{align*}
By the uniform boundedness $\|T_{u_0^\delta}\| \leq \|u_0^\delta\|_{L^\infty}$ and $T_{u_0^\delta} \xrightarrow{s} T_{u_0}$, we obtain
$T_{u_0^\delta}^n \xrightarrow[s]{} T_{u_0}^n,$
which implies
$$ (X^* - (-1)^n(n+1)t\,T_{u_0^\delta}^n- z\mathrm{Id})^{-1} \xrightarrow[s]{\delta \to 0} (X^* - (-1)^n(n+1)t\,T_{u_0}^n- z\mathrm{Id})^{-1}. $$
The strong convergence of the resolvents and the formula \eqref{eq:Pi-ZD-n} allow us to conclude that, for every $t \in \mathbb{R}$, $ZD[u_0^\delta](t)$ converges weakly to $ZD[u_0](t)$ in $L^2(\mathbb{R}).$
\end{proof}

To identify the formula \eqref{eq:BOnTn} in Proposition \ref{prop:BOhier-weak} with the geometric expression in Theorem \ref{thm:zerolimit}, we need an explicit description of the operator $T_{u_0}^n$ in the rational case. The next lemma is the natural generalization of Lemma \ref{lemma:Tu0Tu02}.

\begin{lemma}\label{lemma:Tun-general}
Let $u_0(y)$ be as in Lemma \ref{lemma:Tu0Tu02}. 
For $n \geq 1$ and $f \in L^2_+(\mathbb{R})$, we have the following inductive formula:
\begin{equation}\label{eq:Tun-general}
T_{u_0}^n f(y) = u_0(y)^n f(y) - \sum_{k=0}^{n-1} u_0(y)^{n-1-k} \sum_{j=1}^N \frac{c_j (T_{u_0}^k f)(p_j)}{y - p_j},
\end{equation}
where $N$ is the number of poles of the function $u_0$.
\end{lemma}
\begin{proof}
We proceed by induction on $n \geq 1$. The first identity in Lemma \ref{lemma:Tu0Tu02}
gives the base case when $n = 1$: 
\begin{equation}\label{eq:base_rule}
T_{u_0} f(y) = u_0(y) f(y) - \sum_{j=1}^N \frac{c_j f(p_j)}{y - p_j}.
\end{equation}
Assume the formula holds for some $n \ge 1$. We compute $ T_{u_0}^{n+1} f(y) = T_{u_0} (T_{u_0}^n f)(y). $
Applying the formula \eqref{eq:base_rule}, we have 
\begin{equation}\label{eq:step1}
	T_{u_0}^{n+1} f(y) = u_0(y) (T_{u_0}^n f)(y) - \sum_{j=1}^N \frac{c_j (T_{u_0}^n f)(p_j)}{y - p_j}.
\end{equation}
Now, we substitute the induction hypothesis for $(T_{u_0}^n f)(y)$ into the first term of \eqref{eq:step1}
\begin{align*}
T_{u_0}^{n+1} f(y) 
= & u_0(y) \Big( u_0(y)^n f(y) - \sum_{k=0}^{n-1} u_0(y)^{n-1-k}  \sum_{j=1}^N \frac{c_j (T_{u_0}^k f)(p_j)}{y - p_j} \Big)
 -  \sum_{j=1}^N \frac{c_j (T_{u_0}^n f)(p_j)}{y - p_j}\\
= & u_0(y)^{n+1} f(y) - \sum_{k=0}^{n-1} u_0(y)^{n-k} \sum_{j=1}^N \frac{c_j (T_{u_0}^k f)(p_j)}{y - p_j} - \sum_{j=1}^N \frac{c_j (T_{u_0}^n f)(p_j)}{y - p_j}. 
\end{align*}
We absorb the last term on the right-hand side into the second term,
$$ T_{u_0}^{n+1} f(y) = u_0(y)^{n+1} f(y) - \sum_{k=0}^{n} u_0(y)^{n-k} \sum_{j=1}^N \frac{c_j (T_{u_0}^k f)(p_j)}{y - p_j} . $$
This is exactly the formula \eqref{eq:Tun-general} for $n+1$, which completes the induction.
\end{proof}

\begin{remark}\label{rqreccurence}
For each $k \ge 1$ and each pole $p_j$, the quantity $(T_{u_0}^k f)(p_j)$ in the formula \eqref{eq:Tun-general} is well defined. Indeed, set $g(y) := T_{u_0}^{k-1} f(y)$, then the formula \eqref{eq:Tu_0fpj} in the proof of Lemma \ref{lemma:Tu0Tu02} yields
\begin{equation*}
T_{u_0}^k f (p_j) = c_j (T_{u_0}^{k-1} f)'(p_j) + \sum_{l \neq j} \frac{c_l (T_{u_0}^{k-1} f(p_j)  - T_{u_0}^{k-1} f(p_l) )}{p_j - p_l} + \sum_{l=1}^N \frac{\overline{c_l} T_{u_0}^{k-1} f(p_j) }{p_j - \overline{p_l}}.
\end{equation*}
We see that $T_{u_0}^k f (p_j)$ depends on the derivatives of $f$ at the poles up to order $k$.
\end{remark}

\begin{remark}
When $n=2$, formula \eqref{eq:Tun-general} becomes
$$T_{u_0}^2 f(y) = u_0(y)^2 f(y) - u_0(y) \sum_{j=1}^N \frac{c_j f(p_j)}{y - p_j} - \sum_{j=1}^N \frac{c_j T_{u_0} f(p_j)}{y - p_j},$$
which is exactly the second formula in Lemma \ref{lemma:Tu0Tu02}. 
\end{remark}

With Proposition \ref{prop:BOhier-weak} and Lemma \ref{lemma:Tun-general} in hand, we can now prove Theorem \ref{thm:zerolimit}. Since the argument follows a similar structure as in the third-order case, we only emphasize the modifications specific to the general $n$-th flow.

\begin{proof}[Proof of Theorem \ref{thm:zerolimit}]
We use the same initial datum $u_0(y)$ with simple poles $p_j$ and coefficients $c_j$ in  \eqref{eq:initial-datum}. The characteristic equation for the $n$-th order hierarchy is given by
\begin{align}\label{eq:BOnchara}
 y - (-1)^n (n+1)t \, u_0(y)^n = z, 
\end{align}
which corresponds to a polynomial equation of degree $2nN + 1$,
$$ y Q_0^n (y) - (-1)^n (n+1)t P_0^n(y) = z Q_0^n (y). $$	
For $z= x \in \mathbb{R}$, this equation has an odd number $2\ell+1$ of real roots. Except for a finite set of points $x$, these roots are simple and denoted by $y_0(t,x) < y_1(t,x) < \cdots < y_{2\ell}(t,x)$. Defining $G(y;z) := y - (-1)^n (n+1)t u_0(y)^n - z$, the derivative at a simple real root is 
$$ \partial_y G(y_k(t,x);x) = 1 - (-1)^n n(n+1)t\,u_0(y_k(t,x))^{n-1}u_0'(y_k(t,x)) \neq 0, \; 0 \leq k \leq nN. $$
The sign of this derivative alternates, namely,  $\operatorname{sign}(\partial_y G(y_k;x)) = (-1)^k$.
By the implicit function theorem and the Cauchy--Riemann equations, for $z = x + i\varepsilon$ with $\varepsilon > 0$ sufficiently small, the imaginary parts of the roots $y_k(t,z)$ satisfy
\begin{equation*}
	\text{Im } (y_{k}(t,z))	
	\begin{cases}
		> 0, \qquad k \text{ even}\\
		< 0, \qquad k \text{ odd}.
	\end{cases}
\end{equation*}
For such a complex number $z$, the formula 
\eqref{eq:BOnTn} gives
$$
\Pi ZD[u_0](t, z) = \frac{1}{2 \ii \pi }\, I_+ \Bigl[(X^* - (-1)^n	(n+1)t\,T_{u_0}^n - z\,\mathrm{Id})^{-1} (\Pi u_0)\Bigr].
$$
We now define the function $g_{t,z} \in L^2_+(\mathbb{R})$ as
$$ g_{t,z} := (X^* - (-1)^n (n+1)t\,T_{u_0}^n - z\,\mathrm{Id})^{-1} \Pi u_0. $$
Applying the identity $X^* f(y) = y f(y) + \frac{1}{2 \ii \pi } I_+(f) $, we have
\begin{align*}
	\Pi u_0 &= (X^* - (-1)^n (n+1)t\,T_{u_0}^n - z\,\mathrm{Id}) g_{t,z}(y) \\
	&= (y - z) g_{t,z}(y) + \lambda(t,z) - (-1)^n (n+1)t\, T_{u_0}^n g_{t,z}(y),
\end{align*}
where $\lambda(t,z) = \frac{1}{2 \ii \pi } I_+(f_{t,z})$.
Using Lemma \ref{lemma:Tun-general}, we substitute the expansion for $T_{u_0}^n f_{t,z}$ into the above equation,
\begin{align*}
	\Pi u_0 
	= &  (y - z) f_{t,z}(y) + \lambda(t,z) \\
	& - (-1)^n (n+1)t \left[ u_0(y)^n g_{t,z}(y) - \sum_{k=0}^{n-1} u_0(y)^{n-1-k} \sum_{j=1}^N \frac{c_j T_{u_0}^k g_{t,z}(p_j)}{y - p_j} \right] \\
	= & \left[ y - z - (-1)^n (n+1)t u_0(y)^n \right] g_{t,z}(y) + \lambda(t,z) \\
	& +  \sum_{k=0}^{n-1} u_0(y)^{n-1-k} \sum_{j=1}^N \frac{\tilde{c}_j T_{u_0}^k g_{t,z}(p_j)}{y - p_j}.
\end{align*}
where $\tilde{c}_j = (-1)^n (n+1)t c_j$ and the expressions of $T_{u_0}^k f(p_j)$ are given in Remark \ref{rqreccurence}.
Rearranging the terms, we find that $g_{t,z}(y)$ is a rational function of the form
\begin{align}\label{eq:f_tz_formula}
g_{t,z}(y) 
= \frac{u_0(y) - \lambda(t,z) - \sum_{k=0}^{n-1} u_0(y)^{n-1-k} \sum_{j=1}^N \frac{\mu_{j,k}(t,z)}{y - p_j}}{ y - z - (-1)^n (n+1)t u_0(y)^n },
\end{align}
where the coefficients $\mu_{j,k}(t,z)$ are related to the values $T_{u_0}^k f_{t,z}(p_j)$ and the coefficients $\tilde{c}_j$. This implies that the numerator of \eqref{eq:f_tz_formula} must vanish at all roots of the characteristic equation \eqref{eq:BOnchara} that lie in the upper half-plane. As previously established, for $z=x+i\varepsilon$, these are the $nN+1$ roots $\{y_{2k}\}_{k=0}^{nN}$. Evaluating the numerator at these points yields the following $(nN+1) \times (nN+1)$ linear system
\begin{equation*}
	\lambda (t, z)
	+ \sum_{k=0}^{n-2} u_0(y_{2k})^{n-1-k} \sum_{j=1}^N \frac{\mu_{j,k}(t, z)}{y_{2k} - p_j}
	+ \sum_{j=1}^N \frac{\mu_{j,n-1}(t, z)}{y_{2k} - p_j}
	=  u_0 (y_{2k}),
\end{equation*}
for $k = 0, 1, \dots, nN$. This system generalizes \eqref{eq:lambdatz3} to the $n$-th order case. By Cramer's rule, we can express $\lambda(t,z)$ as the ratio of two determinants:
$$
\lambda (t,z) = \frac{\mathcal{N}}{\mathcal{D}},
$$
where 
\begin{align}
   & \mathcal{D}
	=
	\det\Big(
	\mathcal{D}_{0,1},\dots,\mathcal{D}_{0,N},\;
	\mathcal{D}_{1,1},\dots,\mathcal{D}_{1,N},\;
	\dots,\;
	\mathcal{D}_{n-1,1},\dots,\mathcal{D}_{n-1,N},  \mathcal{D}_0
	\Big), \label{eq:detNDn} \\
	& \mathcal{N}
	=
	\det\Big(
	\mathcal{D}_{0,1},\dots,	\mathcal{D}_{0,N},\;
	\mathcal{D}_{1,1},\dots,	\mathcal{D}_{1,N},\;
	\dots,\;
	\mathcal{D}_{n-1,1},\dots,	\mathcal{D}_{n-1,N}, \mathcal{N}_0
	\Big), \label{eq:detNDn1}
\end{align}
with
\begin{align*}
	& \mathcal{D}_0 = (1,1,\dots,1)^T \in \mathbb{C}^{(nN +1) \times 1},\quad  \mathcal{N}_0 = \big( u_0(y_0), u_0(y_2), \dots, u_0(y_{2nN}) \big)^T \in \mathbb{C}^{(nN +1) \times 1}, \\
	& \mathcal{D}_{m,\nu} =
	\left(
	\frac{u_0^m(y_0) }{y_0-p_\nu}, \frac{u_0^m(y_2) }{y_2-p_\nu},
	\dots,
	\frac{ u_0^m(y_{2 nN})}{y_{2nN}-p_\nu}
	\right)^T  \in \mathbb{C}^{(nN +1) \times 1},  0 \leq  m \leq n-1,\  1 \leq \nu \leq N.
\end{align*}

Before proceeding with the proof, we present a generalized version of Proposition \ref{prop:ND_n=2} to express the determinants $\mathcal{D}$ and $\mathcal{N}$ in terms of Vandermonde-type determinants.

\begin{proposition}\label{prop:ND_N}
Let $u_0$ be as in Theorem \ref{thm:zerolimit}. For fix $z \in \mathbb{C}$, let $\{y_{2k}\}_{k=0}^{2N}$ be zeros in the upper half-plane of the equation \eqref{eq:BOnchara}, as above.

Then, up to a constant $K$, the determinants $\mathcal{D}$ and $\mathcal{N}$ given in \eqref{eq:detNDn} and \eqref{eq:detNDn1} can be reduced the following determinants, respectively,
	\begin{align*}
		& \mathcal{D} =
		\frac{K}{\prod_{k=0}^{nN} \prod_{\nu=1}^N ( \alpha t v_k^n + z - p_\nu)}
		\det \mathcal{V}(v_0, \cdots, v_{nN})
		, \\
		&	\mathcal{N}
		=
		\frac{K}{\prod_{k=0}^{nN} \prod_{\nu=1}^N ( \alpha t v_k^n + z - p_\nu)} \det \widetilde{\mathcal{V}}(v_0, \cdots, v_{nN}),
	\end{align*}
where $v_k : = u_0(y_{2k})$, $\alpha = (-1)^n (n+1)$ for $n\geq 2$, and $\mathcal{V}$ is the standard Vandermonde matrix and $\widetilde{\mathcal{V}}$ a missing power Vandermonde 
\begin{align}\label{eq:vandern}
	\mathcal{V}
	= 
	\begin{pmatrix}
		1 & v_{k} & \dots &  v_{k}^{nN-1} &  v_{k} ^{nN} \\
	\end{pmatrix}_{k=0}^{nN},\quad
	\widetilde{\mathcal{V}}
	= 
	\begin{pmatrix}
		1 & v_{k} & \dots &  v_{k}^{nN-1} &  v_{k} ^{nN+1} \\
	\end{pmatrix}_{k=0}^{nN}.
\end{align}
\end{proposition}

\begin{proof}
The proof is a straightforward generalization of the case $n=2$ (Proposition~\ref{prop:ND_n=2}), so we only indicate the main steps and point out the necessary modifications.

First, for $0 \leq k \leq nN$, we write the $k$-th row of the determinant $\mathcal{D}$ in the form
	\begin{align*}
		\Big(
		& \frac{1}{\alpha t v_k^n + z - p_1},\dots,\frac{1}{\alpha t v_k^n + z - p_N},\;
		\frac{v_k}{\alpha t v_k^n + z - p_1},\dots,\frac{v_k}{\alpha t v_k^n + z - p_N},\; \\
		& \dots,\;
		 \frac{v_k^{n-1}}{\alpha t v_k^n + z - p_1},\dots,\frac{v_k^{n-1}}{\alpha t v_k^n + z - p_N}, 1
		\Big) .
	\end{align*}
Let us define the polynomials (in the formal variable $v$),
\begin{align*}
&Q(v) := \prod_{\nu=1}^N (\alpha t v^n + z - p_\nu), \\
& P_{m, \nu} (v) := v^m \prod_{\mu \ne \nu}^N (\alpha t v^n + z - p_\mu), \;0 \leq  m \leq n-1, 1 \leq \nu \leq N.
\end{align*}
These polynomials will be evaluated at $v = v_k$. 

Multiplying the $k$-th row by $Q(v_k)$, we obtain
	$$
	\mathcal D
	=
	\frac{1}{\prod_{k=0}^{nN} Q(v_k)}
	\det\Big(
	P_{0,1}(v_k), \dots, P_{0,N}(v_k),\;
	\dots,\;
	P_{n-1,1}(v_k), \dots, P_{n-1,N}(v_k),\;
	Q(v_k)
	\Big)_{k=0}^{nN}.$$
We note that $\deg Q(v) = nN$, and $\deg P_{m, \nu} (v) = m + n (N-1) \leq  nN -1$, for $0 \leq  m \leq n-1$,
so these polynomials have degree at most $nN$, and since the number of columns is exactly $nN+1$, these polynomials form a basis of $\mathbb{C}[v]_{\le nN}$. Expanding these polynomials in the monomial basis $(1,v,\dots,v^{nN})$, we obtain a factorization
	\begin{align*}
		& \Big(
		P_1(v_k),\dots,P_N(v_k),\;
		v_kP_1(v_k),\dots,v_k P_N(v_k), \cdots, \\
		&  v_k^{n-1}P_1(v_k),\dots,v_k^{n-1} P_N(v_k)
		Q(v_k)
		\Big)_{k=0}^{nN} 
		= \mathcal V(v_0,\dots,v_{nN}) \cdot \mathrm{Coe}_D,
	\end{align*}
where $\mathcal V$ is the Vandermonde matrix and $\mathrm{Coe}_D$ is the coefficient matrix. Therefore,
	\begin{align*}
		\mathcal D
		= &
		\frac{1}{\prod_{k=0}^{nN} \prod_{\nu=1}^N (\alpha t v_k^n + z - p_\nu)}
		\det \mathcal V(v_0,\dots,v_{nN}) \cdot \det(\mathrm{Coe_D}),
	\end{align*}
In order to treat the coefficient matrix $\mathrm{Coe}_D$, we observe that each polynomial $P_{m,\nu}(v)$ only contains monomials of degree $m \pmod n$, while $Q(v)$ contains only powers divisible by $n$, and its highest-degree term is $(\alpha t)^N v^{nN}$. After reordering the monomial basis like in the Proposition \ref{prop:ND_n=2} , the coefficient matrix becomes block upper triangular, with $n$ identical $N\times N$ diagonal blocks coming from the polynomials $P_{m,\nu}$ and one $1\times1$ block $(\alpha t)^N$ coming from polynomial $Q$. Consequently,
	\begin{align*}
		\det(\mathrm{Coe_D}) 
		& =  \pm  
		\begin{vmatrix}
			N \times N \; \text{block} & 0 &\cdots & 0 & *\\
			\cdots &  \cdots & \cdots  & \cdots & * \\
			0 &  0&  \cdots & N \times N \; \text{block}  & * \\
			0 & 0 & \cdots &  0 & (\alpha t)^{N}\\
		\end{vmatrix}\\
		& = \pm (\alpha t)^{N} \det (N \times N \; \text{block})^n := K
	\end{align*}
where all $N\times N$ block is formed by the coefficients of $\prod_{\mu\neq \nu} (\alpha t v^n + z - p_\mu)$, and the sign comes from permutations of the basis.
	
For the determinant $\mathcal N$, the argument is exactly the same, except that the last column $Q(v_k)$ is replaced by $v_k Q(v_k)$. The latter has degree $nN+1$ and leading term $(\alpha t)^N v^{nN+1}$, while all other columns still have degree at most $nN-1$. Therefore the monomial basis becomes
$
	(1, v, \dots, v^{nN-1}, v^{nN+1}),$
which produces the missing-power Vandermonde matrix $\widetilde{\mathcal V}$.
Using the same block decomposition, the highest-degree coefficient is still $(\alpha t)^N$. Hence,
$\det(\mathrm{Coe}_N) = \det(\mathrm{Coe}_D) = K.
$
This proves the stated formula for $\mathcal N$ and completes the proof.
\end{proof}

We now return to the proof of Theorem \ref{thm:zerolimit}. 

First, by applying the same argument as in Lemma \ref{lemma:ratio-Vander}, we obtain the identity
\begin{equation*}
	\det \widetilde{\mathcal{V}}(v_0, \cdots, v_{nN})= \sum_{k=0}^{nN} v_k \det \mathcal{V} (v_0, \cdots, v_{nN}),
\end{equation*}
where $v_k := u_0(y_{2k})$ for $0 \leq k \leq nN$, and the matrices $\mathcal{V}$ and $\widetilde{\mathcal{V}}$ are as defined in \eqref{eq:vandern}.
By Proposition \ref{prop:ND_N}, the ratio $\lambda(t,x)$ simplifies as follows
$$
\lambda(t, x) = \frac{\mathcal{N}}{\mathcal{D}} = \frac{\det \widetilde{\mathcal{V}}(v_0, \cdots, v_{nN})}{\det \mathcal{V}(v_0, \cdots, v_{nN})} = \sum_{k=0}^{nN} v_k = \sum_{k=0}^{nN}u_0(y_{2k}).
$$
Recalling that $
ZD [u_0] (t, x) =\lambda(t, x)+ \overline{ \lambda(t, x)}$, it follows that
\begin{align*}
	ZD [u_0] (t, x)
	 =  \sum_{k=0}^{nN}u_0(y_{2k}(t, x)) + \sum_{k=0}^{nN}u_0( \overline{y_{2k}(t, x)})
	& = 2 \sum_{k=0}^{\ell}u_0(y_{2k}(t, x)) + \sum_{ k =2 \ell+1}^{2nN}u_0(y_{k}(t, x)) \\
    & = \sum_{\substack{0 \le k \le 2nN \\ y_k \in \mathbb{R}}} (-1)^k\, u_0(y_k),
\end{align*}
where we used Lemma \ref{lemma:fracPQn=2} (see Appendix) in the final equality.
This completes the proof of the formula \eqref{eq:zerolimitgeo} for the case when $u_0(y)$ is rational. The extension to the non-rational case follows by a standard density argument as in the proof of Theorem \ref{thm: mainBO3eps} and is therefore omitted.
This ends the proof of Theorem \ref{thm:zerolimit}.
\end{proof}

\begin{remark}\label{remarkvander}
The Proposition \ref{prop:ND_N} requires $n\geq 2$. The polynomial family $\{Q, P_{0,1}, \cdots,  P_{0,N}\}$ is no longer linearly independent, when $n=1$. Indeed, by observing that $Q(v) = (\alpha t v + z - p_\nu) P_{0, \nu} (v)$ for each $\nu$, one has the identity
\begin{align*}
	Q(v) = \frac{1}{N} \sum_{\nu =1}^{N} (\alpha t v + z - p_\nu) P_{0, \nu} (v),
\end{align*}
where $P_{0, \nu} (v) := \prod_{\mu \neq \nu}^N (\alpha t v + z - p_\mu)$.
Hence, $Q(v)$ belongs to the linear span of $\{P_{0, \nu} (v)\}_{\nu =0}^N$. Therefore the coefficient matrix $\mathrm{Coe}_D$ becomes singular and the block triangular structure used in the proof collapses. As a consequence, the Vandermonde reduction does not apply to the case $n=1$, which must be treated separately (it reduces instead to a Cauchy-type determinant). More about the case $n=1$, see \cite{Gerard2025small}.
\end{remark}

\section{Appendix}\label{secappendix}

\begin{lemma}\label{lemma:fracPQn=2}
Let $u_0(y)$ be a rational function with real coefficients, with no pole on the real line,
$$
u_0(y)= \frac{P_0(y)}{Q_0(y)},
$$
where $Q_0(y)$ is a polynomial of degree $2N$ and $P_0(y)$ is a polynomial of degree $2N-1$. Assume the zeros of $Q_0(y)$ are simple and $P_0(y)$ does not vanish at the zeros of $Q_0(y)$.
Given $t\neq0$, denote by $y_0,\dots,y_{4N}$  the $4N+1$ roots of $p(y)=(y-x)Q^2_0(y)-3t\,P^2_0(y)$.
Then
\begin{equation}\label{eq:fracPQn=2}
\sum_{j=0}^{4N} u_0(y_j)= 0.
\end{equation}
More generally, let
$ \tilde p(\tilde y)=(\tilde y-x)Q^n_0(\tilde y)- (-1)^n (n+1) t\,P^n_0(\tilde y), t\neq0,$
and let $\tilde y_0,\dots,\tilde y_{2nN}$ denote the $2nN+1$ roots of $\tilde p(y)$.
Then
$$\sum_{j=0}^{2nN} u_0(\tilde y_j) = 0.$$
\end{lemma}

\begin{proof}
Since the  roots of $Q_0(y)$ are simple,  the rational function $\frac{P_0(y)}{Q_0(y)}$ can be decomposed as
	\begin{align}\label{eq:fracPQ}
		\frac{P_0(y)}{Q_0(y)} = \sum_{a :\, Q_0(a) = 0} \frac{c_a}{\,y-a\,}, \qquad \textrm{with} \quad c_a = \frac{P_0(a)}{Q'_0(a)}. 
	\end{align}
Next, for any $b\in\mathbb{C}$ with $p(b)\neq0$ the standard logarithmic-derivative identity for the monic polynomial $p(y)=\prod_{j=0}^{4N}(y-y_j)$ gives
	\begin{align*}	
		\sum_{j=0}^{4N}\frac{1}{y_j-b} \;=\; -\frac{p'(b)}{p(b)}.
	\end{align*}
Applying the logarithmic-derivative identity above to the decomposition \eqref{eq:fracPQ}, we have
	\begin{align*}
		\sum_{j=0}^{4N} \frac{P_0(y_j)}{Q_0(y_j)}
		= \sum_{j=0}^{4N}\sum_{a:\,Q_0(a)=0}\frac{P_0(a)}{Q_0'(a)}\frac{1}{y_j-a} 
		& = \sum_{a:\,Q_0(a)=0}\frac{P_0(a)}{Q_0'(a)}\sum_{j=0}^{4N}\frac{1}{y_j-a} \\
		& = -\sum_{a:\,Q(a)=0}\frac{P_0(a)}{Q_0'(a)}\frac{p'(a)}{p(a)}.
	\end{align*}
We evaluate $p(a)$ and $p'(a)$ at a zero $a$ of $Q_0$. Since $Q_0(a)=0$, we have
	$$ p(a) = -3t\,P_0(a)^2,$$ 
and differentiating $p$ gives
	$$p'(y)=Q_0(y)^2+2(y-x)Q_0(y)Q_0'(y)-6tP_0(y)P_0'(y),$$
thus,
	$$ p'(a) = -6t\,P_0(a)P_0'(a).$$
Therefore
	$$\sum_{j=0}^{4N}\frac{P_0(y_j)}{Q_0(y_j)}
	= -2\sum_{a:\,Q_0(a)=0}\frac{P_0'(a)}{Q_0'(a)}.$$	
Thus it remains to show
	\begin{align}\label{eq:fracP-Q}
		\sum_{a:\,Q_0(a)=0}\frac{P_0'(a)}{Q_0'(a)} = 0.
	\end{align}
To prove this, we decompose the rational function $\frac{P_0'(y)}{Q_0(y)}$, by applying the claim \eqref{eq:fracPQ},
	$$
	\frac{P_0'(y)}{Q_0(y)}=\sum_{a:\,Q_0(a)=0}\frac{1}{y-a} \frac{P_0'(a)}{Q_0'(a)}.
	$$
Next, expanding the right-hand side in powers of $\frac{1}{y}$ at infinity gives
	$$
	\sum_{a:\,Q_0(a)=0} \frac{1}{y-a} \frac{P_0'(a)}{Q_0'(a)} = \frac{1}{y}\sum_{a:\,Q_0(a)=0}\frac{P_0'(a)}{Q_0'(a)} + O(y^{-2}).
	$$
On the other hand, since
	$\deg P_0' = 2N-2$ and $\deg Q_0=2N$, we have 
	$
	\left|\frac{P'_0(y)}{Q_0(y)}\right| \leq \frac{C}{y^2},$
which forces the coefficient of $\frac{1}{y}$ to vanish and thus yields \eqref{eq:fracP-Q}. This completes the proof of identity \eqref{eq:fracPQn=2}. As the proof for the general case follows by an analogous argument, it is omitted here.
\end{proof}	

\textbf{Acknowledgements.} The second author would like to thank Louise Gassot, Ola M{\ae}hlen for discussions during the writing of this paper. The two authors are supported by the French Agence Nationale de la Recherche under grant number ANR project ISAAC–ANR-23–CE40-0015-01.

\bibliography{mytexbib}
\bibliographystyle{siam}

\end{document}